\documentclass[11pt]{amsart}
\usepackage{fouriernc} 

\usepackage{Definitions}    
\usepackage{PageSetup}      
\usepackage{Environments}   

\usepackage{afterpage}
\usepackage{mathtools}
\usepackage{subcaption}
\usepackage{lscape}

\hypersetup{
 pdfkeywords={latex starter,template},
 pdfauthor={Niles Johnson},
}

\usepackage{tikz-cd} 

\subjclass[2010]{16D90, 18D05, 19D55, 55M20, 55R12}

\usetikzlibrary{decorations.pathreplacing,angles,quotes}
\newcommand{\Z}{\mathbb{Z}}
\newcommand{\xto}{\xrightarrow}
\newcommand{\mbf}{\mathbf}
\newcommand{\Sp}{\operatorname{Sp}}

\DeclareMathOperator{\End}{End}
\DeclareMathOperator{\tr}{tr}
\newcommand{\rsm}[4]{{#1}^{({#2},{#3})}_{#4}} 
\newcommand{\lsm}[4]{\,^{({#2},{#3})}{#1}_{#4}} 
\newcommand{\rbm}[4]{{#1}^{#2}_{#4}} 
\newcommand{\lbm}[4]{\,^{#2}{#1}_{#4}}  
\newcommand{\rbmm}[5]{{#1}^{{#2},{#5}}_{{#4}}} 

\def\shvar#1#2{{\ensuremath{%
  \hspace{1mm}\makebox[-1mm]{$#1\langle$}\makebox[0mm]{$#1\langle$}\hspace{1mm}%
  {#2}%
  \makebox[1mm]{$#1\rangle$}\makebox[0mm]{$#1\rangle$}%
}}}
\def\sh{\shvar{}}

\title{Topological Hochschild Homology and Higher Characteristics}
\date{\today}
\author{Jonathan A. Campbell}
\address{Department of Mathematics, Vanderbilt University,
1326 Stevenson Center, Nashville, Tennessee, 37240}
\author{Kate Ponto} 
\address{Department of Mathematics, University of Kentucky, 719 Patterson Office Tower, Lexington, Kentucky, 40506}

\begin{document}

\begin{abstract}
	We show that an important classical fixed point invariant, the Reidemeister trace, arises as a topological Hochschild homology transfer.  This generalizes a corresponding classical result for the Euler characteristic and is a first step in showing the Reidemeister trace is in the image of  the cyclotomic trace.  The main result follows from developing  the relationship between shadows \cite{p:thesis}, topological Hochschild homology, and Morita invariance in bicategorical generality. 
\end{abstract}
\maketitle

\setcounter{tocdepth}{1}
\tableofcontents

\section{Introduction}

Many of the technical achievements of modern homotopy theory and algebraic geometry are motivated by questions arising from   fixed point theory. Lefschetz's fixed point theorem is an incredibly successful application of cohomology theory, and it provides the intuition for Grothendieck's development of \'{e}tale cohomology, via the Weil conjectures.  Building on the Riemann-Roch theorem, the Atiyah-Singer index theorem \cite{atiyah_singer} is in essence also a fixed point theorem. In each of these theorems, the goal is to obtain geometric information about fixed points from cohomological information. In this paper, we begin to relate the cyclotomic trace to fixed point theory, with topological Hochschild homology playing the role of the cohomology theory. 

The most basic cohomological invariant of a self-map $f\colon X \to X$ is the Lefschetz number; it is a sort of twisted Euler characteristic.  The Lefschetz number detects fixed points, but it is not a complete invariant.   For that we need a more powerful invariant: the Reidemeister trace, defined as follows. Let $\{x_1, \dots, x_n\}$ be the set of fixed points of $f$. We say $x_i$ and $x_j$ are in the same fixed point class if there is a path $\gamma$ from $x_i$ to $x_j$ such that $\gamma \simeq f(\gamma)$ relative $\{x_i, x_j\} =\{f(x_i), f(x_j)\}$. This is an equivalence relation which partitions the set of fixed points into fixed point classes, and the free abelian group on fixed point classes is denoted $\mbf{Z}[\pi_1 (X)_f]$. The Reidemeister trace of $f$ is  $R(f) = \sum_{x_i} \operatorname{ind}(x_i) [x_i] \in \mbf{Z}[\pi_1 X_f].$ We then have $L(f) = \sum_{x_i} \operatorname{ind}(x_i)$. The Reidemeister trace is a more refined invariant than the Lefschetz number since it supports a converse to the Lefschetz fixed point theorem \cite{shi,wecken}.

From the perspective of homotopy theory,  this description of the Reidemeister trace is unsatisfying.  There are many reasons for this.  One is that in this formulation the Reidemeister trace appears to be a 
strange combination of unstable and stable data. 
This can be resolved by recognizing that  the Reidemeister trace 
is a map of spectra 
\[
  S \to \Sigma^\infty_+ \mc{L} X^f,
\] 
where $\mc{L} X^f$ is the space of paths $x \to f(x)$ \cite{p:thesis, p:coincidence}. 

Experience with algebraic $K$-theory makes the above formulation of the Reidemeister trace very suggestive. Algebraic $K$-theory is a universal receptacle for Euler characteristics \cite{waldhausen, barwick, blumberg_gepner_tabuada}, and it comes equipped with the ``cyclotomic trace'' map $K(R) \to \thh (R)$, where the target is an invariant known as topological Hochschild homology \cite{bokstedt_hsiang_madsen}. For a topological space $X$, the algebraic $K$-theory of $X$ is defined to be $K(\Sigma^\infty_+ \Omega X)$, and $\pi_0 K(\Sigma^\infty_+ \Omega X)$ contains a canonical element $[X]$ corresponding to $X$. It is a folk theorem that the composition
 \begin{equation*}
  S \xrightarrow{[X]} K(\Sigma^\infty_+ \Omega X) \xrightarrow{\tr} \thh (\Sigma^\infty_+ \Omega X) \simeq \Sigma^\infty_+ \mc{L} X \to S 
  \end{equation*} 
is the Euler characteristic. 

The appearance of the loop space and Euler characteristic strongly suggests that the ``twisted Euler characteristic'' $R(f)$ should arise in a very similar way, and there should be corresponding higher traces. In future work we show that indeed, $R(f)$ is in the image of some cyclotomic trace. The main step in showing that is completed in this paper.

\begin{thm}\label{thm:lm_generalization}
  Let $X$ be a topological space homotopy equivalent to a finite CW-complex. The Reidemeister trace is naturally equivalent to the $\thh$ transfer
  \[
    \thh(\operatorname{Mod}^c_S) \to  \thh(\operatorname{Mod}^c_{\Sigma^\infty_+ \Omega X})\to  \thh(\operatorname{Mod}^c_{\Sigma^\infty_+ \Omega X}; F).
  \]
\end{thm}

In this statement $\operatorname{Mod}^c_A$ is the category of compact $A$ modules. The object $\thh(\operatorname{Mod}^c_A; F)$ is a twisted variant of $\thh$ (\cref{right_twisted}). 

The equivalence referenced in \cref{thm:lm_generalization} is induced by Morita equivalences, which are maps
\[
 \thh(A) \xto{\sim} \thh(\operatorname{Mod}^c_A).
\]
In this direction the map is not hard to define, but the homotopy inverse is far less obvious.  It would be desirable to know the inverse. We give a reasonably description of the inverse, and give a very explicit description on $\pi_0$.

For a ring spectrum $A$, an endomorphism $f\colon M \to M$ of a compact $A$-module spectrum determines a map 
$S\to \operatorname{End}(M) $.  Composing with the inclusion of the zero skeleton defines a map 
\[
  S\to \operatorname{End}(M) \to \thh(\operatorname{Mod}^c_A) 
\]
and so an element $[f] \in \pi_0 \thh(\operatorname{Mod}^c_A)$. This sets up the second main theorem of this paper. 

\begin{thm}\label{thm:what_happens_on_pi}
 The image of $[f] \in \pi_0 \thh(\operatorname{Mod}^c_A)$ under the Morita invariance isomorphism 
 \[
   \thh(\operatorname{Mod}^c_A) \simeq \thh(A)
  \]
  is the bicategorical trace of $f$. In particular, for a module $M \in \operatorname{Mod}^c_A$, the image of $[\operatorname{id}_M]$ is $\chi(M)$. 
\end{thm}

The bicategorical trace is defined in \cref{defn:trace}.

As indicated above, all of these invariants are generalizations of the Euler characteristic and, less obviously, they share many formal properties. This observation provides a conceptually clean and very general approach to both \cref{thm:lm_generalization,thm:what_happens_on_pi}:  duality, shadows and traces in bicategories \cite{p:thesis,ps:bicat} exactly capture the relevant properties of the Euler characteristic and its generalizations.  Then these theorems are special cases of far more general results that are proven without any reference to a particular bicategory.

The relevant bicategorical theoretic machinery is developed or recalled in the body of the paper. The key foundational concepts are
\begin{itemize}
  \item base change objects (see \cref{def:base_change}). 
  \item the trace (see \cref{defn:trace})
  \item the Euler characteristic (see \cref{defn:euler_characteristic})
  \item Morita equivalences (see \cref{morita_equivalence})
\end{itemize}
Every theorem in this paper studies the interplay between some of these ingredients. For the convenience of the reader, we provide a concordance of these results in \cref{fig:concordance}, so that they may see the logical dependencies. The four boxed theorems at the top of the figure  are the results from which all of the results in this paper follow. The logical progression is one of gradual specialization --- the difficulty is in identifying the correct categorical context for proving the main results, not in the category theory itself. 
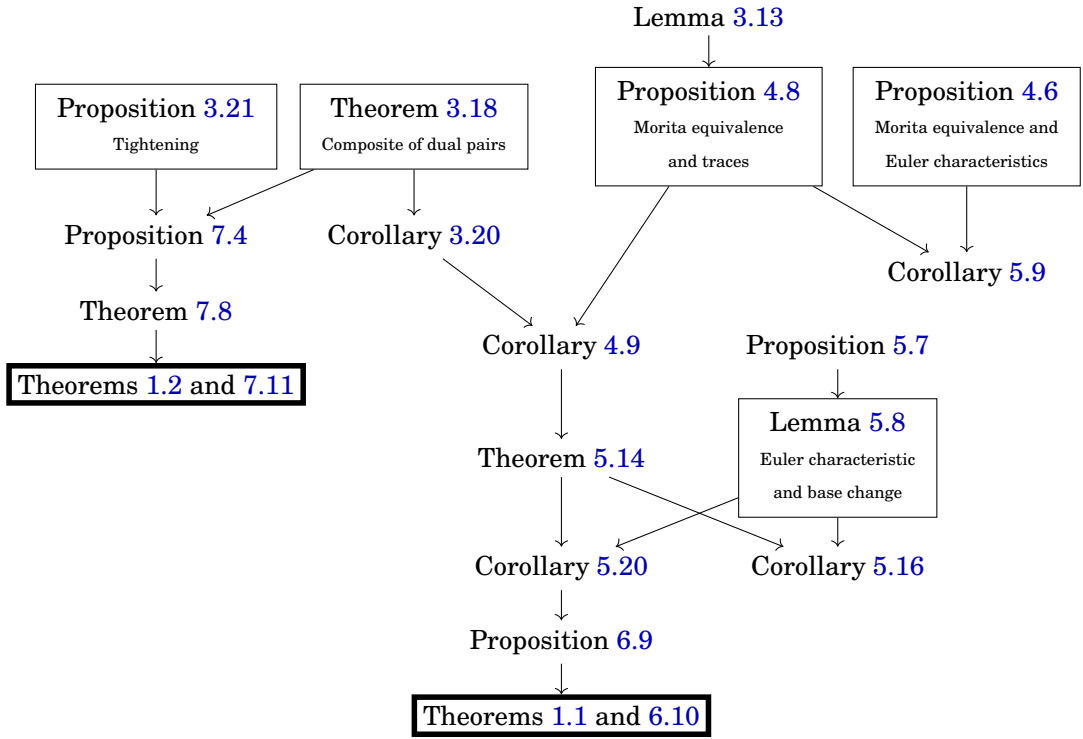
\begin{figure}
 \resizebox{\textwidth}{!}{
\begin{tikzpicture}
\node (313) at (-2.5,22.5) {\cref{rsm_simplified}};
\node (318) [rectangle, draw] at (-6.5,21) {\begin{tabular}{c}\cref{lem:composite_dual_pair}\\{\tiny Composite of dual pairs}\end{tabular}};
\node (320) at (-6.5,19.5) {\cref{euler_char_shadows_3}};
\node (46) [rectangle, draw] at (1,21) {\begin{tabular}{c}\cref{euler_char_shadows}\\{\tiny Morita equivalence and}\\{\tiny  Euler characteristics}\end{tabular}};
\node (48) [rectangle, draw] at (-2.5,21) {\begin{tabular}{c}\cref{euler_char_shadows_2}\\{\tiny Morita equivalence}\\{\tiny  and traces}\end{tabular}};
\node (49) at (-4.5,18) {\cref{euler_char_shadows_4}};
\node (56) at (-.75,18) {\cref{lem:base_change_euler}};
\node (57) [rectangle, draw]  at (-.75,16.5) {\begin{tabular}{c}\cref{lem:base_change_euler_ex}\\ {\tiny Euler characteristic}\\{\tiny  and base change}\end{tabular}};
\node (58) at (1,19) {\cref{cor:base_change_morita_isos}};
\node (514) at (-4.5,16.5) {\cref{base_change_one_object}};
\node (516) at (-.75,15) {\cref{cor:lind_malkiewich}};
\node (519) at (-4.5,15) {\cref{ex:base_change_and_natural_transformation}};
\node (610) at (-4.5,14) {\cref{prop:main_prop}};
\node (611) [rectangle, draw, line width =.75mm]  at (-4.5,13) {\cref{thm:main_theorem,thm:lm_generalization}};
\node (75) at (-10,19.5) {\cref{prop:base_change_euler_monoid}};

\node (79) at (-10,18.5) {\cref{lem:base_change_euler_inverse}};
\node (711) [rectangle, draw, line width =.75mm] at (-10,17.5) {\cref{thm:pi_0_THH,thm:what_happens_on_pi}};

\node (321) [rectangle, draw] at (-10,21) {\begin{tabular}{c} \cref{prop:tightening}\\{\tiny Tightening}\end{tabular}};

\draw [->](57)--(519);
\draw [->](318)--(320);
\draw [->](313)--(48);
\draw [->](320)--(49);
\draw [->](48)--(49);
\draw [->](56)--(57);
\draw [->](46)--(58);
\draw [->](48)--(58);
\draw [->](49)--(514);
\draw [->](57)--(516);
\draw [->](514)--(516);
\draw [->](514)--(519);
\draw [->](519)--(610);
\draw [->](610)--(611);
\draw [->](318)--(75);
\draw [->](75)--(79);
\draw [->](79)--(711);
\draw [->](321)--(75);
\end{tikzpicture}
}
\caption{Concordance of results}\label{fig:concordance}
\end{figure}

\subsection{Outline}

In \cref{sect:thh_spectral} we establish some results about the bicategory of spectral categories, and how $\thh$ behaves on these. We prove that $\thh$ is a shadow in the sense of \cite{p:thesis}. In \cref{sect:duality_and_trace} we define traces and Euler characteristics in bicategories. In the bicategorical context, the Euler characteristic is an invariant of 1-cells, and we establish a number of results about the composition of these characteristics. 

\cref{morita_equivalence} and \cref{sect:euler_char} are devoted to the properties of traces and Euler characteristics under Morita invariance. These are the technical core of the paper, and every result needed to address the main example is treated in great generality in these sections. \cref{morita_equivalence} addresses how traces behave with respect to Morita invariance, while \cref{sect:euler_char} discusses how bicategorical Euler characteristics behave under certain base change maps. 

Having related shadows and $\thh$, in \cref{sect:transfer} we relate the transfer in $\thh$ to the Reidemeister trace. This is achieved by observing that transfers in $\thh$ are nothing more than an example of base change. The results from \cref{morita_equivalence} and \cref{sect:euler_char} then allow us to very explicitly identify certain transfers.

In \cref{sect:cyclotomic} we show that the ``inclusion of objects'' map on $\thh$ is exactly computed by the bicategorical trace, finally relating the two notions of trace that arise in the literature. 

A crucial, but lengthy, computation is relegated to the appendix.

\subsection{Bicategories and Notation}

Here we set our definitions and notations for bicategories. For much more thorough treatments see \cite{benabou,leinster}. 
A {\bf bicategory} $\mc{B}$ consists of objects, $A, B, \ldots$, called $0$-cells,  and categories $\mc{B}(A, B)$ for each pair of objects $A, B$. Objects in the category $\mc{B}(A, B)$ are called 1-cells and morphisms are called 2-cells.  The unit $1$-cell associated to a 0-cell $A$ is denoted $U_A$.
There are horizontal composition functors
\[
\odot \colon  \mc{B}(A, B) \times \mc{B}(B, C) \to \mc{B}(A, C). 
\]
They need not be strictly associative or unital.

The most illuminating examples of bicategories for this paper are: 
\begin{itemize}
  \item The bicategory whose 0-cells are rings and, for rings $A$ and $B$, $\mc{B}(A,B)$ is the category of $(A,B)$-bimodules.  The horizontal composition is the tensor product.
  \item The bicategory whose 0-cells are spaces and, for spaces $A$ and $B$, $\mc{B}(A,B)$ is the category of spaces over $A\times B$.  The  horizontal composition is the pullback 
along the diagonal.  This bicategory also has a stable version \cite{maysig:pht}.
\end{itemize}

\subsection{Acknowledgments} 

This paper should be regarded as a step in manifesting a  perspective linking fixed point theory, $K$-theory, and topological Hochschild homology that has long been known to experts like Randy McCarthy, John Klein, and Bruce Williams. Parts of this perspective have appeared explicitly in the unpublished thesis of Iwashita \cite{iwashita}. 

Campbell thanks Randy McCarthy for a useful conversation about $K$-theory and fixed point theory. He also thanks Ralph Cohen for teaching him the ubiquity and utility of the free loop space. 
Ponto was partially supported by a Simons Collaborations Grant.

We are especially appreciative of the efforts of a very helpful referee who insights significantly improved this paper. 

\section{$\thh$ for Spectral Categories and Shadows}\label{sect:thh_spectral}

In this section we define topological Hochschild homology and review the properties of spectral categories that are useful for our main applications. We show that THH, as an invariant of spectrally enriched categories, is a 
\textit{shadow} in the sense of \cite{p:thesis}. 
This allows us to work in the generality of bicategories, easing and clarifying proofs and simplifying later work.  

As we will make clear, the natural home for $\thh$ is the bicategory of spectral categories. The other familiar property of $\thh$, Morita invariance, is a consequence of this structure. We emphasize that for proving general theorems about $\operatorname{THH}$ almost no other structure is used except for that provided by shadows. There is precedent for this viewpoint in the literature. In \cite{blumberg_mandell} the authors essentially manipulate $\operatorname{THH}$ as a shadow. As another example, in order to explore formal properties of Hochschild homology of DG-categories, Kaledin in \cite{kaledin} defines ``trace functors'' and then notes that they are similar to the second-named author's shadows. From this perspective, there is in some sense nothing ``special'' about $\thh$. Of course, its main property is that it receives a map from algebraic $K$-theory, but we are not yet using that structure.

Topological Hochschild homology is defined at varying levels of generality: it can be defined for ring spectra \cite{bokstedt}, rings with a bimodule coordinate \cite{dundas_mccarthy}, spectral categories and spectral categories with a bimodule coordinate \cite{blumberg_mandell}. For the moment, we work in the generality of spectral categories and bimodules. We begin by considering spectral categories enriched in either symmetric or orthogonal spectra \cite{mandell_may_shipley_schwede}.

\begin{defn} 
	A spectral category $\mc{C}$ is \textbf{pointwise cofibrant} if each mapping spectrum $\mc{C}(a, b)$ is cofibrant in the enriching category.
\end{defn}

\begin{defn}
	Let $\mc{C}$ be a spectral category. Then a $\mc{C}$\textbf{-module} is a spectral functor $\mc{C} \to \Sp$. 
\end{defn}

\begin{defn}
  A \textbf{$(\mc{C}, \mc{D})$-bimodule} is a functor $\mc{M}\colon  \mc{C}^{\text{op}} \sma \mc{D} \to \Sp$. That is, $\mc{M}$ is a collection of spectra $\mc{M}(c, d)$ for $c \in \mc{C}$, $d \in \mc{D}$ together with maps
  \begin{align*}
    &\mc{C}(c, c') \sma \mc{M}(c', d) \to \mc{M}(c, d)\\
    &\mc{M}(c, d) \sma \mc{D}(d,d') \to \mc{M}(c, d') 
  \end{align*}
  An $(\mc{C}, \mc{D})$-bimodule is {\bf pointwise cofibrant} if $\mc{M}(c, d)$ is cofibrant. A {\bf morphism} of $(\mc{C}, \mc{D})$-bimodules $\mc{M} \to \mc{M}'$ is a collection of maps $\mc{M}(c, d) \to \mc{M}'(c, d)$ which commute with the appropriate structure. 
\end{defn}

\begin{rmk}
	Note that this has the opposite variance of what is standard for bimodules in the literature. This convention seems to be more useful for bookkeeping for us. 
\end{rmk}

\begin{defn}
  Let $\mc{C}$ be a pointwise cofibrant spectral category, and $\mc{Q}$ a $(\mc{C},\mc{C})$-bimodule. The \textbf{topological Hochschild homology of $\mc{C}$ with coefficients in $\mc{Q}$} is the geometric realization of a spectrum whose $n$th simplicial level is
  \[
  \thh(\mc{C}; \mc{Q})_n \coloneqq N^{\text{cy}}_n (\mc{C}, \mc{Q}) = \bigvee_{c_0, \dots, c_n} \mc{C}(c_0, c_1) \sma \mc{C}(c_1, c_2) \sma \cdots \sma \mc{C}(c_{n-1}, c_n) \sma \mc{Q}(c_n, c_0)
  \]
  That is, 
  \[
  \thh(\mc{C};\mc{Q}) \coloneqq \vert\thh(\mc{C};\mc{Q})_\bullet\vert
  \]
\end{defn}

\begin{rmk}
	The reader is warned that most literature makes a distinction between B\"{o}kstedt's construction $\thh(\mc{C};\mc{Q})$ and $N^{\text{cy}}(\mc{C};\mc{Q})$. When $\mc{C}$ is pointwise cofibrant, they are equivalent.  Since we will work on the level of homotopy categories, and ignore questions of equivariance, we therefore ignore the distinction. 
\end{rmk}

Spectral bimodules may be manufactured from functors. This is an example of what we later call \textit{base change}. 

\begin{defn}\label{def:base_change}
  Let $F\colon  \mc{A} \to \mc{C}$ and  $G\colon  \mc{B} \to \mc{C}$ be functors between spectral categories. Define an $(\mc{A}, \mc{B})$-bimodule $ \ _F \mc{C}_G$ as follows. For objects $a\in \mc{A}$ and $b \in \mc{B}$
  \[
  \ _F \mc{C}_G (a, b) \coloneqq \mc{C}(F(a), G(b)).  
  \]
 The right action of $\mc{B}$ on $\ _F \mc{C}_G$ is given by the functor $G$:
  \begin{align*}
  \ _F \mc{C}_G(a, b) \sma \mc{B} (b, b') &= \mc{C}(F(a),G(b)) \sma \mc{B} (b, b') \to \mc{C}(F(a), G(b)) \sma \mc{C}(G(b), G(b'))\\ &\to \mc{C} (F(a), G(b')) = \ _F \mc{C}_G (a, b') 
  \end{align*}  The left action of $\mc{A}$ is similar.
\end{defn}

\begin{example}
	When $F\colon  \mc{C} \to \mc{C}$ is an endofunctor and $G = \id$, we can form $\ _F \mc{C}$, and similarly $\mc{C}_F$. 
\end{example}

\begin{defn}\label{right_twisted}
  Let $\mc{C}$ be a pointwise cofibrant spectral category and $F\colon  \mc{C} \to \mc{C}$ be an endofunctor. We defined the \textbf{right twisted topological Hochschild homology} to be
  \[
  \thh(\mc{C}; F)\coloneqq\thh(\mc{C}; \mc{C}_F) 
  \]
\end{defn}

\begin{rmk}
When $F = \id$, we recover $\thh(\mc{C})$. 
\end{rmk}

\begin{example}\label{eg:twisted_thh}
A good example to keep in mind is the following. Let $A$ be a commutative ring spectrum,  $P$ be an $A$-module, and let $\operatorname{Mod}^c_A$ denote the category of compact $A$-modules. Consider the functor $-\wedge_AP \colon  \operatorname{Mod}^c_A \to \operatorname{Mod}^c_A$ given by $M \mapsto M \sma_A P$, where $M \in \operatorname{Mod}^c_A$. We show  in \cref{ex:twist_thh_morita} that the twisted $\thh(\operatorname{Mod}^c_A; -\wedge _AP)$ coincides with $\thh(A; P)$.
\end{example}

Topological Hochschild homology is clearly functorial in the bimodule coordinate so that given a map of $(\mc{C}, \mc{C})$-bimodules $\mc{Q} \to \mc{Q}'$  there is a map
\[
\thh(\mc{C}; \mc{Q}) \to \thh(\mc{C}; \mc{Q}').  
\]
Furthermore, if $\mc{A} \to \mc{C}$ is a map, then we get an induced map $\thh(\mc{A}) \to \thh(\mc{C})$. There is also a refinement of both \cite{blumberg_mandell}. Let $F\colon  \mc{A} \to \mc{C}$ be a map of spectral categories and let $\mc{Q}$ be a $(\mc{C},\mc{C})$-bimodule. Then there is a map
\[
\thh(\mc{A};{_F\mc{Q}_F}) \to \thh(\mc{C}, \mc{Q}) 
\]
and if there is a map $\mc{P} \to {_F \mc{Q}_F}$ we obtain
\[
\thh(\mc{A}; \mc{P}) \to \thh (\mc{C}, \mc{Q}). 
\]

We now describe the bicategory structure on the category of spectral categories. First, we note some homotopical properties of spectral categories. In the sequel, we work with a bicategory enriched in various homotopy categories; the following remarks establish that we may do this.

To begin, we have the following rephrasing of \cite[Prop.~6.1]{schwede_shipley} found in \cite[Prop.~2.4]{blumberg_mandell}.

\begin{prop}
The category, $\operatorname{Mod}_{(\mc{C}, \mc{D})}$ of $(\mc{C}, \mc{D})$-bimodules forms a closed model category with object-wise weak equivalences.
\end{prop}

For any small spectral category $\mc{C}$, we have the following rephrasing of \cite[Prop.~6.3]{schwede_shipley} due to \cite[Prop.~2.7, Prop.~2.8]{blumberg_mandell}.

\begin{prop}
	Given a small spectral category $\mc{C}$ there is an endofunctor $Q\colon  \mc{C}\text{at}_{Sp} \to \mc{C}\text{at}_{Sp}$ such that $Q \mc{C}$ is pointwise cofibrant and there is a map $Q \mc{C} \to \mc{C}$ that is a pointwise weak equivalence. Furthermore, if $\mc{M}$ is a cofibrant $(\mc{C}, \mc{D})$-module, then $\mc{M}$ is pointwise cofibrant. 
\end{prop}

Furthermore, by the remark following \cite[Prop.~3.6]{blumberg_mandell}, if $\mc{C}$ is pointwise cofibrant, and $\mc{P} \to \mc{P}'$ is a weak equivalence of spectral categories,
 then the induced map 
 \[
	\thh(\mc{C}, \mc{P}) \to \thh (\mc{C}, \mc{P}')
 \]
is a weak equivalence.  Thus, for instance, if $Q\mc{P} \to \mc{P}$ is a cofibrant replacement of $\mc{P}$, $\thh(\mc{C}, Q\mc{P}) \to \thh(\mc{C}, \mc{P})$ is a weak equivalence.

These propositions imply that can move  between models and replace bimodules by weakly equivalent ones at will. Given this, we work on the level of homotopy categories.

\begin{defn}
  The bicategory of small spectral categories is the bicategory whose objects are pointwise cofibrant small spectral categories, and whose morphism categories are 
\[
 \operatorname{Ho}\left(\operatorname{Mod}_{(\mc{C}, \mc{D})}\right)
\]
 for pointwise cofibrant small spectral categories $\mc{C}$ and $\mc{D}$.
\end{defn}

\begin{rmk} 
  The composition of 1-cells is defined as follows. Let $\mc{M}$ be an $(\mc{C}, \mc{D})$-bimodule and $\mc{N}$ an $(\mc{D},\mc{E})$-bimodule. Then we may form an $(\mc{C}, \mc{E})$-bimodule $\mc{M} \odot \mc{N}$
  \begin{align*}
    (\mc{M} \odot \mc{N})(c,e) &\coloneqq \mc{M}(c,-) \sma^{\mbf{L}}_{\mc{D}} \mc{N}(-,e)\\
    &\coloneqq B(\mc{M}(c,-), \mc{D}, \mc{N}(-,e))
  \end{align*}
  where $B(-,-,-)$ denotes the two-sided bar construction. 
  
  This descends to
  \[
  \operatorname{Ho}\left(\Mod_{(\mc{C}, \mc{D})}\right) \times \operatorname{Ho}\left(\Mod_{(\mc{D}, \mc{E})}\right) \to \operatorname{Ho}\left(\Mod_{(\mc{C}, \mc{E})}\right)
  \]
  Checking that this is associative is straightforward but tedious. One explicitly writes out the bar construction and cofibrantly replaces as needed.  The composition of 2-cells is the composition of natural transformations. 
\end{rmk} 

As a cyclic bar construction, $\thh$ has cyclic invariance built into it.  This cyclic invariance is also present in Hochschild homology and is an essential part of  the Hattori-Stallings  trace 
\[
K_0(A)\to HH_0(A).
\] 
There is a general categorical setup due to the second named author \cite{p:thesis} that encodes exactly the kind of properties that $\thh$ enjoys as a functor of spectral categories.

\begin{defn}[\cite{p:thesis}]\label{def:shadow}
  Let $\mc{B}$ be a bicategory.  A \textbf{shadow functor} for $\mc{B}$
  consists of functors
  \[
   \sh{-}\colon  \mc{B}(C,C) \to \mbf{T}
  \]
  for each object $C$ of $\mc{B}$ and some fixed category $\mbf{T}$, equipped
  with a natural isomorphism
  \[
   \theta\colon \sh{M\odot N}\xto{\cong} \sh{N\odot M}
  \]
  for $M\in \mc{B}(C, D)$ and $N\in \mc{B}(D, C)$ such that the following
  diagrams commute whenever they make sense:
  \[\xymatrix{
    \sh{(M\odot N)\odot P} \ar[r]^\theta \ar[d]_{\sh{a}} &
    \sh{P \odot (M\odot N)} \ar[r]^{\sh{a}} &
    \sh{(P\odot M) \odot N}
    \\
    \sh{M\odot (N\odot P)} \ar[r]^\theta 
    & \sh{(N\odot P)\odot M} \ar[r]^{\sh{a}} 
   & \sh{N\odot (P\odot M)}\ar[u]_\theta 
  }\]
  \[\xymatrix{
    \sh{M\odot U_C} \ar[r]^\theta \ar[dr]_{\sh{r}} &
    \sh{U_C\odot M} \ar[d]^{\sh{l}} \ar[r]^\theta &
    \sh{M\odot U_C} \ar[dl]^{\sh{r}}
    \\
    &\sh{M}}\]
\end{defn}

Note that   if $\sh{-}$ is a shadow functor on $\mc{B}$, then the composite
  \[\xymatrix{
    \sh{M\odot N} \ar[r]^\theta &
    \sh{N\odot M} \ar[r]^\theta &
    \sh{M\odot N}
   }\]
  is the identity \cite[Prop.~4.3]{ps:bicat}.

\begin{thm}\label{thm:thh_shadow}
  Topological Hochschild Homology is a shadow. That is, it gives a family of functors
  \[
  \thh (-) \colon  \operatorname{Ho}\left(\operatorname{Mod}_{(\mc{C}, \mc{C})}\right) \to \operatorname{Ho}(\operatorname{Sp}) 
  \]
  that satisfy the required properties. 
\end{thm}

\begin{proof}

  The main property of shadows is that for $(\mc{C}, \mc{D})$-bimodule $\mc{M}$ and a $(\mc{D}, \mc{C})$-bimodule $\mc{N}$, there is an isomorphism
  \[
  \theta \colon  \sh{\mc{M} \odot \mc{N}}  \to \sh{\mc{N} \odot \mc{M}}. 
  \]
  Unpacking this into the usual notation, this is equivalent to the demand that there is an isomorphism 
  \[
  \theta\colon  \thh (\mc{C}, B(\mc{M}, \mc{D}, \mc{N})) \to \thh(\mc{D}, B(\mc{N}, \mc{C}, \mc{M})) 
  \]
  However, this is the classical Dennis-Morita-Waldhausen argument \cite[Prop.~6.2]{blumberg_mandell} --- in this case, there is an isomorphism of underlying point-set spectra.

  The commutativity of the rest of the diagrams follow from essentially the same argument. 
\end{proof}

\section{Duality and  trace}\label{sect:duality_and_trace}

In the previous section, we showed that $\thh$ is an example of a shadow on a bicategory. This is quite a general notion, and many bicategories possess shadows. In addition, in any bicategory with shadow, one can define a notion of trace, which one can think of as a vast generalization of the trace of an endomorphism in a symmetric monoidal category. In this section we recall the definitions required to define a trace and collect the results about the trace that we will need below. It is at this point that we begin to work in bicategorical generality. 

As a starting point it is useful to have a few bicategories in mind. 
The following two examples are very important for our intended applications.  Let $\mc{V}$ be a symmetric monoidal category. (In what follows, one can imagine that $\mc{V}$ is the category of spectra.) 
\begin{enumerate}
  \item   Let $\mc{B}(\mathbf{Mon}(\mc{V}))$ be the bicategory whose  objects are monoids in $\mc{V}$, 1-cells are bimodules over monoids in $\mc{V}$ and whose 2-cells are maps of bimodules.
  \item  Let $\mc{B}(\mathbf{Cat}(\mc{V}))$ be the bicategory whose objects are categories enriched in $\mc{V}$, whose 1-cells are functors $F \colon  \mc{C}^{\text{op}} \otimes \mc{D} \to \mc{V}$ (also known as $(\mc{C}, \mc{D})$-bimodules) and whose 2-cells are natural transformations of such. 
\end{enumerate}
These bicategories both have shadows that take values in $\mc{V}$.

\begin{example}\label{bcat_example}
Let $\mc{V} = \mbf{Sp}$. Then $\mc{B}(\mathbf{Cat}(\mc{V}))$ is the bicategory of spectral categories. Because of the homotopical issues outlined above we work in the full subcategory of pointwise cofibrant small spectral categories. In general, if $\mc{V}$ and $\mbf{Cat}(\mc{V})$ have some kind of homotopical structure, we understand $\mc{B}(\mbf{Cat}(\mc{V}))$ to be modified in order to give the homotopically correct definitions. 
\end{example}

\begin{rmk}
  Note that every monoid in $\mc{V}$ can be made into a $\mc{V}$-category with one object, giving an embedding $\mbf{Mon}(\mc{V}) \to \mbf{Cat}(\mc{V})$; this is the enriched version of the usual embedding $\mbf{Mon} \to \mbf{Cat}$. Thus, at the level of bicategories, we have an embedding
  \[
  \mc{B}(\mbf{Mon}(\mc{V})) \to \mc{B}(\mbf{Cat}(\mc{V})).  
  \]
\end{rmk}

The following definition is at the core of all of the constructions in this paper. 

\begin{defn}We say that a 1-cell $M\in \mc{B}(C,D)$ in a bicategory is \textbf{right dualizable} if there is a 1-cell $N\in \mc{B}(D,C)$, called its \textbf{right dual}, and coevaluation and evaluation 2-cells $\eta\colon U_C \to M\odot N$ and $\epsilon \colon N\odot M\to U_D$ satisfying the triangle identities.
We  say that $(M,N)$ is a \textbf{dual pair}, that $N$ is \textbf{left dualizable}, and that $M$ is its \textbf{left dual}.
\end{defn}

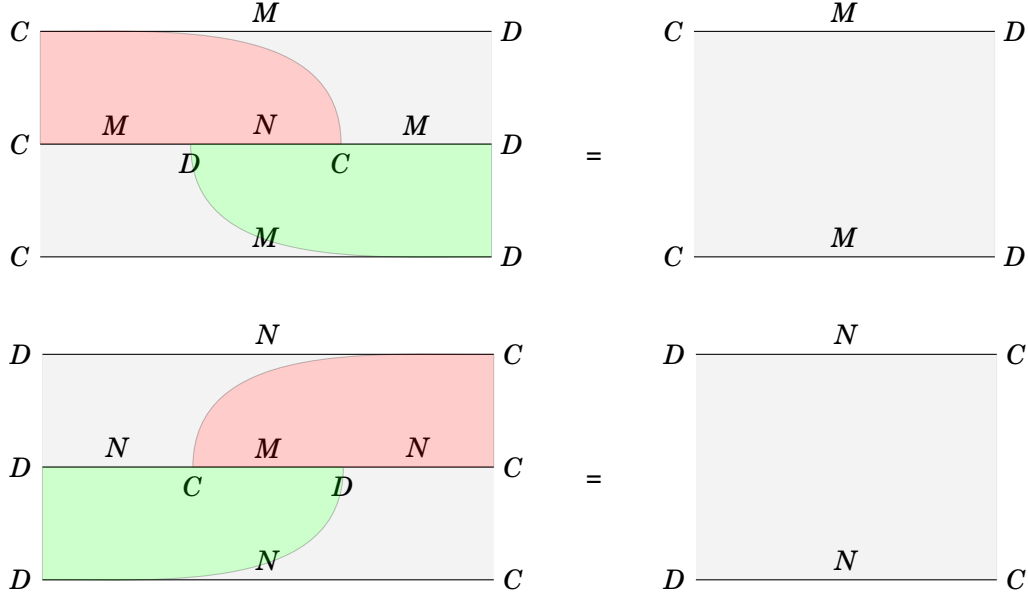
\begin{figure}
\begin{tikzpicture}
\coordinate (A1) at (0,3);
\node at (A1)  [left]{$C$};
\coordinate (B1) at (6,3);
\node at (B1) [right] {$D$};
\coordinate (A2) at (0,1.5);
\node at (A2) [left] {$C$};
\coordinate (B2) at (2,1.5);
\node at (B2) [below] {$D$};
\coordinate (A3) at (4,1.5);
\node at (A3)  [below]{$C$};
\coordinate (B3) at (6,1.5);
\node at (B3) [right] {$D$};
\coordinate (A4) at (0,0);
\node at (A4)[left] {$C$};
\coordinate (B4) at (6,0);
\node at (B4) [right] {$D$};

\draw (A1)--node[midway, above] {$M$}(B1);
\draw (A2)--node[midway, above] {$M$}(B2);
\draw (A3)--node[midway, above] {$M$}(B3);
\draw (A4)--node[midway, above] {$M$}(B4);
\draw (B2)--node[midway, above] {$N$}(A3);
\draw[fill=red, opacity=0.2] (A1)to [out = 0, in = 90] (A3)--(A2)--(A1);
\draw[fill=green, opacity=0.2] (B2)--(B3)--(B4)to [out =180, in =-90](B2);

\draw[fill=black, opacity=0.05] (A1)--(B1)--(B3)--(A3)to[out =90, in =0](A1);
\draw[fill=black, opacity=0.05]  (A2)--(B2) to [out = -90, in = 180] (B4)--(A4);
\end{tikzpicture}
\hspace{.5cm}
\raisebox{1.5cm}{=}
\hspace{.5cm}
\begin{tikzpicture}
\coordinate (A1) at (0,3);
\node at (A1)  [left]{$C$};
\coordinate (B1) at (4,3);
\node at (B1) [right] {$D$};
\coordinate (A4) at (0,0);
\node at (A4)[left] {$C$};
\coordinate (B4) at (4,0);
\node at (B4) [right] {$D$};

\draw (A1)--node[midway, above] {$M$}(B1);
\draw (A4)--node[midway, above] {$M$}(B4);
\draw[fill=black, opacity=0.05]  (A1)--(B1)--(B4)--(A4);
\end{tikzpicture}

\vspace{.5cm}
\begin{tikzpicture}
\coordinate (B1) at (0,3);
\node at (B1) [left] {$D$};
\coordinate (A1) at (6,3);
\node at (A1)[right] {$C$};

\coordinate (B2) at (0,1.5);
\node at (B2)[left] {$D$};
\coordinate (A2) at (2,1.5);
\node at (A2)[below] {$C$};
\coordinate (B3) at (4,1.5);
\node at (B3) [below] {$D$};
\coordinate (A3) at (6,1.5);
\node at (A3)[right] {$C$};

\coordinate (B4) at (0,0);
\node at (B4)[left] {$D$};
\coordinate (A4) at (6,0);
\node at (A4)[right] {$C$};

\draw (B1)--node[midway, above] {$N$}(A1);
\draw (B2)--node[midway, above] {$N$}(A2);
\draw (A2)--node[midway, above] {$M$}(B3);
\draw (B3)--node[midway, above] {$N$}(A3);
\draw (B4)--node[midway, above] {$N$}(A4);

\draw[fill=red, opacity=0.2] (A1)to [out = 180, in = 90] (A2)--(A3)--(A1);
\draw[fill=green, opacity=0.2] (B3)--(B2)--(B4)to [out =0, in =-90](B3);

\draw[fill=black, opacity=0.05] (A1)--(B1)--(B2)--(A2)to[out =90, in =180](A1);
\draw[fill=black, opacity=0.05]  (A3)--(B3) to [out = -90, in = 0] (B4)--(A4);
\end{tikzpicture}
\hspace{.5cm}
\raisebox{1.5cm}{=}
\hspace{.5cm}
\begin{tikzpicture}
\coordinate (A1) at (0,3);
\node at (A1)  [left]{$D$};
\coordinate (B1) at (4,3);
\node at (B1) [right] {$C$};
\coordinate (A4) at (0,0);
\node at (A4)[left] {$D$};
\coordinate (B4) at (4,0);
\node at (B4) [right] {$C$};

\draw (A1)--node[midway, above] {$N$}(B1);
\draw (A4)--node[midway, above] {$N$}(B4);
\draw[fill=black, opacity=0.05]  (A1)--(B1)--(B4)--(A4);
\end{tikzpicture}
\caption{Pasting diagrams for dual pairs.
}\label{fig;dual_pairs}
\end{figure}

\begin{rmk}
While we will not use them as a formal proof, some of the results in the next sections have illuminating graphical descriptions as pasting diagrams.  

In our pasting diagrams vertices represent 0-cells, edges represent 1-cells, and colored regions represent 2-cells.  Since we will need to eventually make circular diagrams (\cref{fig:pasting_trace}), we do not identify 0-cells when they are the same.  Instead we rely on the convention that vertices in consecutive layers that align should be regarded as the same.
We suppress unit isomorphisms and many unit 1-cells.  Pale gray regions are identity 2-cells.

See \cref{fig;dual_pairs} for pasting diagrams for a dual pair.
\end{rmk}

\begin{example} 
For rings $C$ and $D$, an $(C,D)$-bimodule $M$ is right dualizable if it is finitely generated and projective as a right $D$-module.  
A choice of dual is the $(D,C)$-bimodule of right $D$-module homomorphisms $\Hom_D(M,D)$.  Note that $\Hom_D(M,D)$ is regarded as a $(D,C)$-bimodule using the left $C$-module structure on $M$ and left $D$-module structure on $D$. 

The coevaluation map is the composite  
\[
  C\to \Hom_D(M,M)\xleftarrow{\sim} M\otimes_D\Hom_D(M,D)
\] 
where
the second map is an isomorphism since $M$ is  finitely generated and projective as an $D$ module. 
The evaluation map for this dual pair is the  evaluation map 
\[
  \Hom_D(M,D)\otimes_C M\to D.
\]

Dually, $M$ is left dualizable if it is finitely generated and projective as a left $C$-module. 
\end{example}

\begin{example}  
Costenoble-Waner duality \cite[Chapter~18]{maysig:pht} is a special case of the duality theory above and generalizes Spanier-Whitehead and Atiyah duality.  

The parameterized stable homotopy category $\operatorname{Ex}$ of \cite{maysig:pht} has a fiberwise suspension spectrum functor  from the bicategory of fibered spaces (without sections).  If we regard a closed smooth manifold $X$, or compact ENR, as a space over $\ast\times X$  its fiberwise suspension spectrum is dualizable with dual the desuspension of the fiberwise one point compactification of the normal bundle \cite[18.5.1]{maysig:pht}. 
\end{example}

Using this definition and that of a shadow we can define traces of 2-cells associated to dualizable 1-cells. The following definition will be crucial for the constructions below.

\begin{defn}\cite{p:thesis}\label{defn:trace}
  Let $\mc{B}$ be a bicategory with a shadow functor and $(M,N)$ be a dual pair.
  The \textbf{trace} of a 2-cell $f\colon Q\odot M\rightarrow M\odot P$ is the composite:
  \[
   \sh{Q}\cong \sh{Q\odot U_C}\xto{\sh{\id_Q\odot\eta}}
  \sh{Q\odot M\odot N} \xto{\sh{f\odot\id_N}}
  \sh{M\odot P\odot N} \xto{\theta}
  \sh{N\odot M\odot P} \xto{\sh{\epsilon \odot \id_P}}
  \sh{U_D\odot P}\cong   \sh{P}.
  \]
The trace of a 2-cell $g\colon N\odot Q\to P\odot N$ is 
  \[
   \sh{Q}\cong \sh{U_C\odot Q}\xto{\sh{\eta\odot\id_Q}}
  \sh{M\odot N\odot Q} \xto{\sh{\id_M\odot g}}
  \sh{M\odot P\odot N} \xto{\theta}
  \sh{P\odot N\odot M} \xto{\sh{\id_P \odot \epsilon}}
    \sh{P\odot U_D}\cong   \sh{P}.
  \]
\end{defn}

\begin{rmk}
After applying the shadow, we glue together vertical edges to form a bullseye diagram as in \cref{fig:pasting_trace}.  As above, we do collapse most 0-cells.  
In these diagrams we read 2-cells as directed from the innermost circle to the outermost circle and 1-cells clockwise.

Once we have applied the shadow we compose 2-cells by stacking circles.
\end{rmk}

\begin{figure}
\begin{tikzpicture}
\coordinate (A1) at (1,0);
\node at (A1) [right]{$C$};
\node (Q1) at (-1,0)[left]{$Q$};
\draw (0,0) circle (1cm);
\filldraw (1,0) circle (1pt);

\draw (0,0) circle (2cm);
\filldraw (2,0) circle (1pt);
\filldraw ({2*cos(360/3)},{2*sin(360/3)}) circle (1pt);
\filldraw ({2*cos(2*360/3)},{2*sin(2*360/3)}) circle (1pt);

\coordinate (A2) at ({2*cos(360/3)},{2*sin(360/3)});
\node  at (A2)[below right]{$C$};
\coordinate (A3) at ({2*cos(2*360/3)},{2*sin(2*360/3)});
\node at (A3) [above right]{$C$};
\coordinate (B1) at (2,0);
\node  at (B1)[right]{$D$};

\node (M1) at  ({2*cos(360/6)},{2*sin(360/6)})[above right]{$M$};
\node (N1) at  ({2*cos(5*360/6)},{2*sin(5*360/6)})[below  right]{$N$};
\node (Q2) at (-2,0)[left]{$Q$};

\draw (0,0) circle (3cm);
\filldraw (3,0) circle (1pt);
\filldraw ({3*cos(360/3)},{3*sin(360/3)}) circle (1pt);
\filldraw ({3*cos(2*360/3)},{3*sin(2*360/3)}) circle (1pt);

\coordinate (B2) at ({3*cos(360/3)},{3*sin(360/3)});
\node at (B2) [below right]{$D$};
\coordinate (A5)  at ({3*cos(2*360/3)},{3*sin(2*360/3)});
\node at (A5)[above right]{$C$};
\coordinate (B3) at (3,0);
\node at (B3) [right]{$D$};

\node (P1) at  ({3*cos(360/6)},{3*sin(360/6)})[above right]{$P$};
\node (N2) at  ({3*cos(5*360/6)},{3*sin(5*360/6)})[below  right]{$N$};
\node (M2) at (-3,0)[left]{$M$};

\draw (0,0) circle (4cm);
\filldraw ({4*cos(2*360/3)},{4*sin(2*360/3)}) circle (1pt);

\coordinate (B6) at ({4*cos(2*360/3)},{4*sin(2*360/3)});
\node at (B6) [below right]{$D$};
\coordinate (B4) at ({4*cos(2*360/3)},{4*sin(2*360/3)});

\node (P2) at  ({4*cos(360/6)},{4*sin(360/6)})[above right]{$P$};

\draw [fill=red, opacity=0.2](A1)to [out = 90, in =0] (A2) arc(120:-120:2cm)--(A3)to [out = 0, in =-90] (A1);
\draw [fill=blue, opacity=0.2](A3)--(A5) arc(-120:-240:3cm)--(B2)arc(120:0:3cm)--(B3)--(B1)arc(0:240:2cm)--(A3);
\draw [fill=green, opacity=0.2](B4)to[out = 150, in = -170,looseness = 1.2 ] (B2)arc(120:360:3cm)--(B3)to [out = -70, in =-30,looseness = 1.2 ](B4);

\draw[fill=black, opacity=0.05] (A1)to [out = 90, in =0] (A2)arc (120:240:2cm) (A3)to [out = 0, in =-90](A1) arc (0:-360:1cm);
\draw[fill=black, opacity=0.05] (B1)--(B3) arc(0:-120:3cm) (A5)--(A3) arc (240:360:2cm) (B1);
\draw[fill=black, opacity=0.05] (B4)to[out = 150, in = -170,looseness = 1.2 ] (B2)arc(120:0:3cm)--(B3)to [out = -70, in =-30,looseness = 1.2 ](B4) arc(240:600:4cm) (B4);
\end{tikzpicture}
\caption{The trace}\label{fig:pasting_trace}
\end{figure}

The bicategorical trace generalizes both the symmetric monoidal trace \cite{dold_puppe} and the Hattori-Stallings trace \cite{hattori,stallings}.  

\begin{rmk}In \cite{dold_puppe}, there is a particularly elegant proof of the Lefschetz fixed point theorem that  relies on the observations that the fixed point index \cite{dold_idx} is the trace in the stable homotopy category, the Lefschetz number is the trace in the homotopy category of chain complexes, 
and the symmetric monoidal trace is functorial: that is 
\[
  F(\tr(f))=\tr(F(f)).
\]
The Reidemeister trace, in its many variants, is an example of the bicategorical trace \cite{p:thesis}.  Since these are bicategorical traces, the identification of the varied forms of the Reidemeister trace 
is a consequence of the functoriality of the bicategorical trace.
\end{rmk}

Let $(M,N)$ be a dual pair and $Q$ and $P$ be 1-cells so that $N\odot Q\odot M$ and $M\odot P\odot N$ are defined. We then  fix the following notation: 
\begin{gather*}
  \rsm{\eta}{M}{N}{Q}\colon Q\odot M\cong U_C\odot Q\odot M
  \xto{\eta\odot \id_Q\odot \id_M} 
  M\odot N\odot Q\odot M
 \\
  \rsm{\epsilon}{M}{N}{P}\colon M\odot P\odot N\odot M
  \xto{\id_M \odot \id_P \odot\epsilon}
  M\odot P\odot U_D\cong M\odot P
 \\
   \lsm{\eta}{M}{N}{Q} \colon N\odot Q\cong N\odot Q\odot U_C
  \xto{\id_N\odot \id_Q\odot \eta} N\odot Q\odot M\odot N
 \\
   \lsm{\epsilon}{M}{N}{P}\colon N\odot M\odot P\odot N
   \xto{\epsilon\odot \id_P \odot \id_N }
   U_D\odot P\odot N\cong P\odot N
\end{gather*}

 \begin{rmk}\label{left_or_right_for_trace} 
If $(M,N)$ is a dual pair, the {\bf dual} of a map 
$g\colon N\odot Q\to P\odot N$, denoted $g^\star$, is the composite 
\begin{align*}
	Q\odot M\xto{\rsm{\eta}{M}{N}{Q}}
 	M\odot N\odot Q\odot M\xto{\id_M \odot g\odot \id_M}M\odot P\odot N\odot M
	\xto{\rsm{\epsilon}{M}{N}{P} }
	M\odot P
\end{align*}
Since $\tr(g)=\tr(g^\star)$ \cite[Prop.~7.6]{ps:bicat} we freely move between traces defined with respect to $M$ and those defined with respect to $N$.

Note that  $\rsm{\eta}{M}{N}{Q}$ is the dual of $\lsm{\eta}{M}{N}{Q}$,
$\rsm{\epsilon}{M}{N}{Q}$ is the dual of $\lsm{\epsilon}{M}{N}{Q}$,
and 
\begin{gather*}
	\sh{Q}\xto{\tr\left(\rsm{\eta}{M}{N}{Q}\right)=\tr\left(\lsm{\eta}{M}{N}{Q}\right)} \sh{N\odot Q\odot M}
	\quad \text{and } \quad 
	\sh{M\odot P\odot N}\xto{\tr\left(\rsm{\epsilon}{M}{N}{P}\right)=\tr\left(\lsm{\epsilon}{M}{N}{P}\right)} \sh{P}
\end{gather*}
\end{rmk}

\begin{lem}\label{rsm_simplified}
For a dual pair $(M,N)$ and 1-cells $Q$ and $P$ so that $N\odot Q\odot M$ and $M\odot P\odot N$ are defined, $\tr\left(\rsm{\epsilon}MNP\right)$  is 
\[
	\sh{M\odot P\odot N}\xto{\sim}\sh{P\odot N\odot M}\xto{\sh{\id_P\odot \epsilon}} \sh{P\odot U_D}\cong\sh{P}
\]
and 
$\tr\left(\rsm{\eta}MNQ\right)$ is the composite 
\[
\sh{Q}\cong \sh{Q\odot U_C}\xto{\sh{\id_Q\odot \eta}}\sh{Q\odot M\odot N}\cong \sh{N\odot Q\odot M}.
\]
\end{lem}

\begin{proof}
In the trace of $\rsm{\epsilon}{M}{N}{P}$  (\cref{fig:morita_trace}) there is a coevaluation/evaluation pair that can be canceled.  
Cancelling this pair gives the composite above and is illustrated in \cref{fig:morita_trace_2}.  
The proof for $\rsm{\eta}MNQ$ is similar.   
\end{proof}

\begin{figure}
\begin{subfigure}[t]{.63\linewidth}
\begin{tikzpicture}
\coordinate (B1) at (1,0);
\node at (B1) [right]{$D$};
\coordinate (A1) at ({1*cos(1*360/3)},{1*sin(1*360/3)});
\node at (A1) [below right]{$C$};
\coordinate (A2) at ({1*cos(2*360/3)},{1*sin(2*360/3)});
\node at (A2) [above right]{$C$};

\node (Q1) at (-1,0)[left]{$P$};
\node (M1) at  ({1*cos(1*360/6)},{1*sin(1*360/6)})[right]{$M$};
\node (N1) at  ({1*cos(5*360/6)},{1*sin(5*360/6)})[right]{$N$};

\draw (0,0) circle (1cm);
\filldraw (1,0) circle (1pt);
\filldraw ({1*cos(360/3)},{1*sin(360/3)}) circle (1pt);
\filldraw ({1*cos(2*360/3)},{1*sin(2*360/3)}) circle (1pt);

\draw (0,0) circle (2cm);
\filldraw (2,0) circle (1pt);
\filldraw ({2*cos(360/5)},{2*sin(360/5)}) circle (1pt);
\filldraw ({2*cos(2*360/5)},{2*sin(2*360/5)}) circle (1pt);
\filldraw ({2*cos(3*360/5)},{2*sin(3*360/5)}) circle (1pt);
\filldraw ({2*cos(4*360/5)},{2*sin(4*360/5)}) circle (1pt);

\coordinate (A3) at ({2*cos(2*360/5)},{2*sin(2*360/5)});
\node at (A3)[above left]{$C$};
\coordinate (B2) at ({2*cos(1*360/5)},{2*sin(1*360/5)});
\node at (B2)[above]{$D$};
\coordinate (A4) at ({2*cos(0*360/5)},{2*sin(0*360/5)});
\node at (A4)[right]{$C$};
\coordinate (B3) at ({2*cos(4*360/5)},{2*sin(4*360/5)});
\node at (B3)[above]{$D$};
\coordinate (A5) at ({2*cos(3*360/5)},{2*sin(3*360/5)});
\node at (A5)[above right ]{$C$};

\node (Q2) at (-2,0)[left]{$P$};
\node (M2) at  ({2*cos(3*360/10)},{2*sin(3*360/10)})[above]{$M$};
\node (N2) at  ({2*cos(7*360/10)},{2*sin(7*360/10)})[below]{$N$};
\node (M3) at  ({2*cos(9*360/10)},{2*sin(9*360/10)})[below right]{$M$};
\node (N3) at  ({2*cos(1*360/10)},{2*sin(1*360/10)})[above right ]{$N$};

\draw (0,0) circle (3cm);
\coordinate (A6) at ({3*cos(2*360/3-2*360/3+4*360/5)},{3*sin(2*360/3-2*360/3+4*360/5)});
\node at (A6)[above]{$C$};
\coordinate (B4) at ({3*cos(0*360/3-2*360/3+4*360/5)},{3*sin(0*360/3-2*360/3+4*360/5)});
\node at (B4)[above]{$D$};
\coordinate (A7) at ({3*cos(1*360/3-2*360/3+4*360/5)},{3*sin(1*360/3-2*360/3+4*360/5)});
\node at (A7)[left]{$C$};

\filldraw  ({3*cos(2*360/3-2*360/3+4*360/5)},{3*sin(2*360/3-2*360/3+4*360/5)}) circle (1pt);
\filldraw  ({3*cos(1*360/3-2*360/3+4*360/5)},{3*sin(1*360/3-2*360/3+4*360/5)}) circle (1pt);
\filldraw  ({3*cos(0*360/3-2*360/3+4*360/5)},{3*sin(0*360/3-2*360/3+4*360/5)}) circle (1pt);

\node (M4) at ({3*cos(1*360/6-2*360/3+4*360/5)},{3*sin(1*360/6-2*360/3+4*360/5)})[above]{$M$};
\node (N4) at ({3*cos(5*360/6-2*360/3+4*360/5)},{3*sin(5*360/6-2*360/3+4*360/5)})[right]{$N$};
\node (Q3) at ({3*cos(3*360/6-2*360/3+4*360/5)},{3*sin(3*360/6-2*360/3+4*360/5)})[below]{$P$};

\node (Q4) at ({4*cos(3*360/6-2*360/3+4*360/5)},{4*sin(3*360/6-2*360/3+4*360/5)})[below]{$P$};

\draw (0,0) circle (4cm);
\coordinate (A8) at ({4*cos(0*360/3-2*360/3+4*360/5)},{4*sin(0*360/3-2*360/3+4*360/5)});
\node at (A8)[above]{$C$};
\filldraw ({4*cos(0*360/3-2*360/3+4*360/5)},{4*sin(0*360/3-2*360/3+4*360/5)}) circle (1pt);

 \draw [fill=orange, opacity=0.2](B1)to [out = 90, in = -60](B2)arc (1*360/5:-1*360/5:2cm)--(B3)to [out = 60, in =-90] (B1); 
 \draw [fill=blue, opacity=0.2] (A4)arc (0:-2*360/5:2cm)--(A5)to [out = -90, in =200](A6)to[out =20, in=-60] (A4);
 \draw [fill=blue, opacity=0.2] (A6)to [out =-10 , in=-45,looseness =1.2](A8)to [out  =135, in =100,looseness =1.2] (A7) arc(1*360/3-2*360/3+4*360/5:-360/5:3cm)--(A6);
  \draw [fill=black, opacity=0.05](B1)to [out = 90, in = -60](B2)arc (1*360/5:4*360/5:2cm) (B3)to [out = 60, in =-90] (B1) arc (0:-360:1cm) --(B1)
  ; 
   \draw [fill=black, opacity=0.05] (A4)arc (0:3*360/5:2cm)--(A5)to [out = -90, in =200](A6)arc (4*360/5:-360/5:3cm) --(A6) to[out =20, in=-60] (A4);
   \draw [fill=black, opacity=0.05] (A8)arc (-2*360/3+4*360/5:360-2*360/3+4*360/5:4cm)--(A8)to [out =-45 , in=-10,looseness =1.2](A6)arc(-1*360/5:-8*360/15:3cm)(A7)to [out  =100, in =135,looseness =1.2] (A8);
\end{tikzpicture}
\caption{The trace of $\rsm{\epsilon}{N}{M}{P}$} 
\label{fig:morita_trace}
\end{subfigure}
\begin{subfigure}[t]{.3\linewidth}
\begin{tikzpicture}

\draw (0,0) circle (1cm);
\coordinate (A6) at ({1*cos(2*360/3-2*360/3+4*360/5)},{1*sin(2*360/3-2*360/3+4*360/5)});
\node at (A6)[above]{$C$};
\coordinate (B4) at ({1*cos(0*360/3-2*360/3+4*360/5)},{1*sin(0*360/3-2*360/3+4*360/5)});
\node at (B4)[above]{$D$};
\coordinate (A7) at ({1*cos(1*360/3-2*360/3+4*360/5)},{1*sin(1*360/3-2*360/3+4*360/5)});
\node at (A7)[left]{$C$};

\filldraw  ({1*cos(2*360/3-2*360/3+4*360/5)},{1*sin(2*360/3-2*360/3+4*360/5)}) circle (1pt);
\filldraw  ({1*cos(1*360/3-2*360/3+4*360/5)},{1*sin(1*360/3-2*360/3+4*360/5)}) circle (1pt);
\filldraw  ({1*cos(0*360/3-2*360/3+4*360/5)},{1*sin(0*360/3-2*360/3+4*360/5)}) circle (1pt);

\node (M4) at ({1*cos(1*360/6-2*360/3+4*360/5)},{1*sin(1*360/6-2*360/3+4*360/5)})[above]{$M$};
\node (N4) at ({1*cos(5*360/6-2*360/3+4*360/5)},{1*sin(5*360/6-2*360/3+4*360/5)})[right]{$N$};
\node (Q3) at ({1*cos(3*360/6-2*360/3+4*360/5)},{1*sin(3*360/6-2*360/3+4*360/5)})[below]{$P$};

\node (Q4) at ({2*cos(3*360/6-2*360/3+4*360/5)},{2*sin(3*360/6-2*360/3+4*360/5)})[below]{$P$};

\draw (0,0) circle (2cm);
\coordinate (A8) at ({2*cos(0*360/3-2*360/3+4*360/5)},{2*sin(0*360/3-2*360/3+4*360/5)});
\node at (A8)[above]{$C$};
\filldraw ({2*cos(0*360/3-2*360/3+4*360/5)},{2*sin(0*360/3-2*360/3+4*360/5)}) circle (1pt);

 \draw [fill=blue, opacity=0.2] (A6)to [out =-10 , in=-45,looseness =1.2](A8)to [out  =135, in =100,looseness =1.2] (A7) arc(1*360/3-2*360/3+4*360/5:-360/5:1cm)--(A6);
   \draw [fill=black, opacity=0.05] (A8)arc (-2*360/3+4*360/5:360-2*360/3+4*360/5:2cm)--(A8)to [out =-45 , in=-10,looseness =1.2](A6)arc(-1*360/5:-8*360/15:1cm)(A7)to [out  =100, in =135,looseness =1.2] (A8);
\end{tikzpicture}
\caption{After collapsing the coevaluation/evaluation pair} 
\label{fig:morita_trace_2}
\end{subfigure}
\caption{Diagrams for \cref{rsm_simplified}}
\end{figure}

The stable homotopy category is symmetric monoidal and the suspension spectrum of a closed smooth manifold or compact ENR $X$ is dualizable.   The trace of the identity map 
of $X$ is a stable map $S\to S$ and this is the Euler characteristic of $X$ under the identification of stable $\pi_0$ with $\mathbb{Z}$.  As a result, we refer to symmetric monoidal traces
of identity maps and bicategorical traces of identity 2-cells as {\bf Euler characteristics} and denote them by $\chi(X)$. We formalize this in a definition.

\begin{defn}\label{defn:euler_characteristic}
  If $M\in \mc{B}(C,D)$ and $M$ is right dualizable, the \textbf{Euler characteristic} (\cref{fig:euler_char}) is the trace of the identity 2-cell of $M$ and is a map \[\sh{U_C}\to \sh{U_D}.\]   
If $N$ is the right dual of $M$, $\chi(M)=\chi(N)$.
\end{defn}

\begin{rmk}
  Thinking of the Euler characteristic as a \textit{map} rather than a \textit{object} is an important psychological move for working with constructions in the sequel. The importance of this formulation of the Euler characteristic cannot be overstated.
\end{rmk}

\begin{figure}
\begin{tikzpicture}
\coordinate (A1) at (1,0);
\node at (A1) [left]{$C$};
\node (UA1) at (-1,0)[right]{$U_{C}$};

\draw (0,0) circle (1cm);
\filldraw (1,0) circle (1pt);

\draw (0,0) circle (2cm);

\filldraw (-2,0) circle (1pt);
\filldraw (2,0) circle (1pt);
\coordinate (A2) at (2,0);
\node at (A2) [right]{$C$};
\coordinate (B1) at (-2,0);
\node at (B1) [right]{$D$};
\node (M1) at (0,-2)[above]{$M$};
\node (N1) at (0,2)[below]{$N$};

\draw (0,0) circle (3cm);
\filldraw (-3,0) circle (1pt);
\coordinate (B2) at (-3,0);
\node at (B2) [left]{$D$};

\node (UB1) at (3,0)[right]{$U_{D}$};

\draw [fill=red, opacity=0.2](A2)arc(0:360:2cm)--(A2) --(A1) arc(0:-360:1cm)--(A1);
\draw [fill=green, opacity=0.2](B2)arc(-180:-540:3cm)--(B2) --(B1) arc(180:540:2cm)--(B1);
\end{tikzpicture}
\caption{The Euler characteristic}\label{fig:euler_char}
\end{figure}
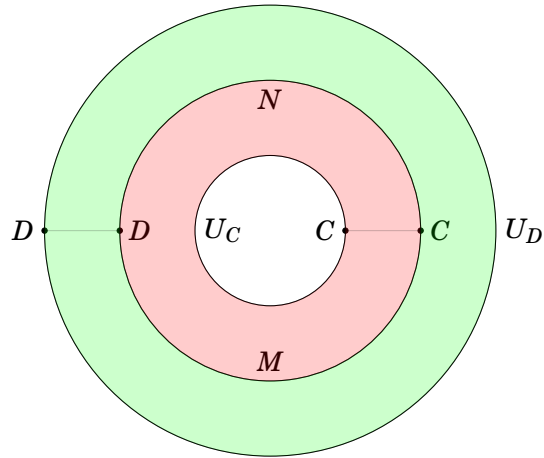

The Euler characteristic is multiplicative on fibrations and its refinements to the Lefschetz number and Reidemeister trace satisfy the appropriate generalizations of multiplicativity \cite{ps:mult}.  These results are consequences of the following very convenient result describing the compatibility between traces and bicategorical composition. It is an easy generalization of the corresponding symmetric monoidal fact
and is an essential foundation for many of the results in the next sections.

\begin{thm} \cite[16.5.1]{maysig:pht}\cite[Prop.~7.5]{ps:bicat} \label{lem:composite_dual_pair}
If $M_1\in \mc{B}(C,D)$ and $M_2\in \mc{B}(D,E)$ are right dualizable, then $M_1\odot M_2$ is right dualizable.  The trace of 
\[
  Q_1\odot M_1\odot M_2\xto{f_1\odot \id_{M_2}}
  M_1\odot Q_2\odot M_2\xto{\id_{M_1}\odot f_2} M_1\odot M_2\odot Q_3
\] is 
\[
\sh{Q_1}\xto{\tr(f_1)}\sh{Q_2}\xto{\tr(f_2)}\sh{Q_3}.
\]
\end{thm}

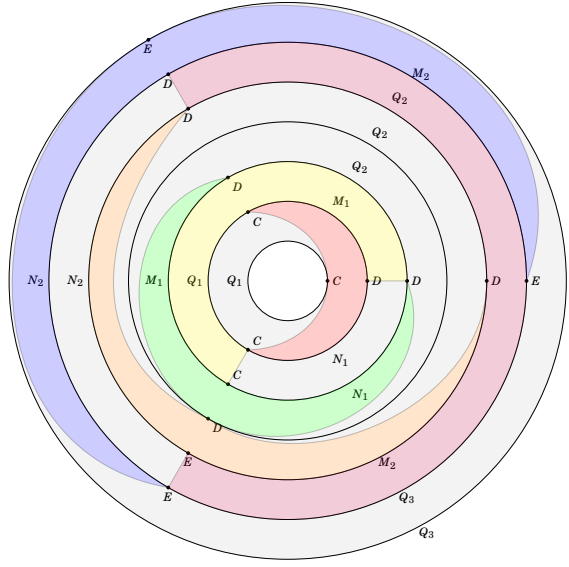
\begin{figure}
\resizebox{.6\textwidth}{!}{
\begin{tikzpicture}
\coordinate (A1) at (1,0);
\node at (A1) [right]{\footnotesize $C$};
\node (Q1) at (-1,0)[left]{\footnotesize $Q_1$};
\draw (0,0) circle (1cm);
\filldraw (1,0) circle (1pt);

\draw (0,0) circle (2cm);
\filldraw (2,0) circle (1pt);
\filldraw ({2*cos(360/3)},{2*sin(360/3)}) circle (1pt);
\filldraw ({2*cos(2*360/3)},{2*sin(2*360/3)}) circle (1pt);

\coordinate (A2) at ({2*cos(360/3)},{2*sin(360/3)});
\node  at (A2)[below right]{\footnotesize $C$};
\coordinate (A3) at ({2*cos(2*360/3)},{2*sin(2*360/3)});
\node at (A3) [above right]{\footnotesize $C$};
\coordinate (B1) at (2,0);
\node  at (B1)[right]{\footnotesize $D$};

\node (M1) at  ({2*cos(360/6)},{2*sin(360/6)})[above right]{\footnotesize $M_1$};
\node (N1) at  ({2*cos(5*360/6)},{2*sin(5*360/6)})[below  right]{\footnotesize $N_1$};
\node (Q2) at (-2,0)[left]{\footnotesize $Q_1$};

\draw (0,0) circle (3cm);
\filldraw (3,0) circle (1pt);
\filldraw ({3*cos(360/3)},{3*sin(360/3)}) circle (1pt);
\filldraw ({3*cos(2*360/3)},{3*sin(2*360/3)}) circle (1pt);

\coordinate (B2) at ({3*cos(360/3)},{3*sin(360/3)});
\node at (B2) [below right]{\footnotesize $D$};
\coordinate (A5)  at ({3*cos(2*360/3)},{3*sin(2*360/3)});
\node at (A5)[above right]{\footnotesize $C$};
\coordinate (B3) at (3,0);
\node at (B3) [right]{\footnotesize $D$};

\node (P1) at  ({3*cos(360/6)},{3*sin(360/6)})[above right]{\footnotesize $Q_2$};
\node (N2) at  ({3*cos(5*360/6)},{3*sin(5*360/6)})[below  right]{\footnotesize $N_1$};
\node (M2) at (-3,0)[left]{\footnotesize $M_1$};

\draw (0,0) circle (4cm);
\filldraw ({4*cos(2*360/3)},{4*sin(2*360/3)}) circle (1pt);

\coordinate (B6) at ({4*cos(2*360/3)},{4*sin(2*360/3)});
\node at (B6) [below right]{\footnotesize $D$};
\coordinate (B4) at ({4*cos(2*360/3)},{4*sin(2*360/3)});

\node (P2) at  ({4*cos(360/6)},{4*sin(360/6)})[above right]{\footnotesize $Q_2$};

\draw (0,0) circle (5cm);
\filldraw (5,0) circle (1pt);
\filldraw ({5*cos(360/3)},{5*sin(360/3)}) circle (1pt);
\filldraw ({5*cos(2*360/3)},{5*sin(2*360/3)}) circle (1pt);

\coordinate (C1) at ({5*cos(2*360/3)},{5*sin(2*360/3)});
\node at (C1) [below]{\footnotesize $E$};
\coordinate (B7) at ({5*cos(360/3)},{5*sin(360/3)});
\node at (B7) [below]{\footnotesize $D$};
\coordinate (B8) at (5,0);
\node at (B8) [right]{\footnotesize $D$};
\node (P3) at  ({5*cos(360/6)},{5*sin(360/6)})[above right]{\footnotesize $Q_2$};
\node (N3) at  ({5*cos(3*360/6)},{5*sin(3*360/6)})[left]{\footnotesize $N_2$};
\node (M3) at  ({5*cos(5*360/6)},{5*sin(5*360/6)})[below]{\footnotesize $M_2$};

\draw (0,0) circle (6cm);
\filldraw (6,0) circle (1pt);
\filldraw ({6*cos(360/3)},{6*sin(360/3)}) circle (1pt);
\filldraw ({6*cos(2*360/3)},{6*sin(2*360/3)}) circle (1pt);

\coordinate (C2) at ({6*cos(2*360/3)},{6*sin(2*360/3)});
\node at (C2) [below]{\footnotesize $E$};
\coordinate (B9) at ({6*cos(360/3)},{6*sin(360/3)});
\node at (B9) [below]{\footnotesize $D$};
\coordinate (C3) at (6,0);
\node at (C3) [right]{\footnotesize $E$};

\node (T1) at ({6*cos(5*360/6)},{6*sin(5*360/6)})[below ]{\footnotesize $Q_3$};
\node (N4) at  ({6*cos(3*360/6)},{6*sin(3*360/6)})[left]{\footnotesize $N_2$};
\node (M4) at  ({6*cos(360/6)},{6*sin(360/6)}) [right]{\footnotesize $M_2$};

\draw (0,0) circle (7cm);
\filldraw ({7*cos(360/3)},{7*sin(360/3)}) circle (1pt);
\coordinate (C4) at ({7*cos(360/3)},{7*sin(360/3)});
\node at (C4) [below]{\footnotesize $E$};
\node (T2) at ({7*cos(5*360/6)},{7*sin(5*360/6)})[below ]{\footnotesize $Q_3$};

\draw [fill=red, opacity=0.2](A1)to [out = 90, in =0] (A2) arc(120:-120:2cm)--(A3)to [out = 0, in =-90] (A1);
\draw [fill=black, opacity=0.05](A1)to [out = 90, in =0] (A2) arc(120:240:2cm)--(A3)to [out = 0, in =-90] (A1) arc (0: -360:1cm);
\draw [fill=yellow, opacity=0.2](A3)--(A5) arc(-120:-240:3cm)--(B2)arc(120:0:3cm)--(B3)--(B1)arc(0:240:2cm)--(A3);
\draw [fill=black, opacity=0.05](A3)--(A5) arc(240:360:3cm)--(B3)--(B1)arc(0:-120:2cm)--(A3);
\draw [fill=green, opacity=0.2](B4)to[out = 150, in = -170,looseness = 1.2 ] (B2)arc(120:360:3cm)--(B3)to [out = -70, in =-30,looseness = 1.2 ](B4);
\draw [fill=black, opacity=0.05](B4)to[out = 150, in = -170,looseness = 1.2 ] (B2)arc(120:0:3cm)--(B3)to [out = -70, in =-30,looseness = 1.2 ](B4) arc(-120: 240:4cm) --(B4);

\draw [fill=orange, opacity=0.2](B6)to [out = 155, in =-130,looseness = 1.2 ] (B7) arc(120:360:5cm)--(B8)to [out = -90, in =-35] (B6);
\draw [fill=black, opacity=0.05](B6)to [out = 155, in =-130,looseness = 1.2 ] (B7) arc(120:0:5cm)--(B8)to [out = -90, in =-35] (B6) arc(-120:240:4cm) --(B6);
\draw [fill=purple, opacity=0.2](C1)--(C2) arc(-120:120:6cm)--(B9)--(B7)arc(120:-120:5cm)--(C1);
\draw [fill=black, opacity=0.05](C1)--(C2) arc(-120:-240:6cm)--(B9)--(B7)arc(120:240:5cm)--(C1);
\draw [fill=blue, opacity=0.2](C4)to[out = -150, in = 170,looseness = 1.2 ] (C2)arc(-120:-360:6cm)--(C3)to [out = 70, in =30,looseness = 1.2 ](C4);
\draw [fill=black, opacity=0.05](C4)to[out = -150, in = 170,looseness = 1.2 ] (C2)arc(240:360:6cm)--(C3)to [out = 70, in =30,looseness = 1.2 ](C4) arc (120:-240:7cm) --(C4);
\end{tikzpicture}
}
\caption{Composite of traces
}\label{fig:composite}
\end{figure}

\cref{fig:composite} is a graphical representation of the composite of traces. 
The trace of $(\id_{M_1}\odot f_2)(f_1\odot \id_{M_2})$ can be visualized by sliding the outer colored segments over the inner light gray segments. 

We will use \cref{lem:composite_dual_pair}  in the following form. Given dual pairs $(M_1,N_1)$ and $(M_2, N_2)$, a 1-cell $L$ so that $M_1\odot L\odot M_2$ is defined,  and an endomorphism 
$Q\odot L\xto{f} L\odot P,$
let $M_1\odot f\odot M_2$ be the composite  
\begin{align*} 
 M_1\odot Q\odot N_1\odot M_1\odot L\odot M_2
 &\xto{\rsm{\epsilon}{M_1}{N_1}{Q} \odot \id_L \odot \id_{M_2}}
 M_1\odot Q\odot L\odot M_2
\\
 & \xto{\id_{M_1} \odot f\odot \id_{M_2}} M_1\odot L\odot P\odot M_2
\\
 &\xto{\id_{M_1}\odot \id_L\odot \rsm{\eta}{M_2}{N_2}{P}} M_1\odot L\odot M_2\odot N_2\odot P\odot M_2.
\end{align*}

\begin{cor}\label{euler_char_shadows_3}
If $L$ is right dualizable 
\[\xymatrix@C=50pt{
\sh{M_1\odot Q\odot N_1}\ar[r]^-{\tr(M_1\odot f\odot M_2)}\ar[d]_-{\tr\left( \rsm{\epsilon}{M_1}{N_1}{Q}\right)}
&\sh{N_2\odot P\odot M_2}
\\
\sh{Q}\ar[r]^-{\tr(f)}&\sh{P}\ar[u]_{\tr\left( \rsm{\eta}{M_2}{N_2}{P}\right)}
}\]
commutes.
\end{cor}

As the definition of $M_1\odot f\odot M_2$ suggests, 
this follows by applying 
\cref{lem:composite_dual_pair}
to the maps
\begin{gather*}
 M_1\odot Q\odot N_1\odot M_1\xto{\rsm{\epsilon}{M_1}{N_1}{Q}} M_1\odot Q
\\
Q\odot L \xto{ f} L\odot P
\\
P\odot M_2\xto{\rsm{\eta}{M_2}{N_2}{P}} M_2\odot N_2\odot P\odot M_2
\end{gather*}

Finally, we record a useful proposition that will be a needed on a few occasions. It is an easy consequence of the formal properties of the trace. 

\begin{prop}\label{prop:tightening}\cite[Prop.~7.1]{ps:bicat} 
Let $M$ be right dualizable, let $f \colon  Q\odot M \to M \odot P$, $g\colon  Q' \to Q$ and $h\colon P \to P'$ be 2-cells. Then
\[
\sh{h}\circ \tr(f)\circ\sh{g}=\tr (\id_M \odot h)\circ f\circ (g\odot \id_M).
\]
\end{prop}

\section{Morita equivalence in bicategories}\label{morita_equivalence}

The Morita invariance of $\thh$ is one of its defining properties, and one of its most useful. If one takes the view that $\thh$ is a shadow on a bicategory, then the Morita invariance becomes a property not of $\thh$ itself, but rather its categorical context. That is, Morita equivalence is the natural notion of an equivalence in a bicategory, so $\thh$ is a Morita invariant simply because it is a bicategorical construct. Since everything we prove about Morita invariance is true at the level of bicategories, we work at that level of generality. This section recalls the definition of a Morita equivalence in a bicategory, and develops the basic properties of such equivalences with respect to trace and Euler characteristic. Since it is a notion of equivalence, the trace and characteristic are essentially insensitive to Morita equivalence, but keeping track of isomorphisms is important for the sequel and future work.

\begin{defn}\label{defn:morita_equivalence}
A pair of one cells $M\in \mc{B}(C,D)$ and $N\in \mc{B}(D,C)$ is a {\bf Morita equivalence} if $(M,N)$ and $(N,M)$ are dual pairs and
the coevaluation and evaluation maps for each dual pair are inverses.  That is, if $\eta_{(M,N)}$ and $\epsilon_{(M,N)}$ are the coevaluation and evaluation 
for $(M,N)$ and $\eta_{(N,M)}$ and $\epsilon_{(N,M)}$ are the coevaluation and evaluation 
for $(N,M)$ then
\begin{align*}
 \eta_{(N,M)}\circ \epsilon_{(M,N)}=\id_{N\odot M} &\qquad \epsilon_{(N,M)}\circ \eta_{(M,N)}=\id_{U_{C}}
 \\
 \eta_{(M,N)}\circ \epsilon_{(N,M)}=\id_{M\odot N}&\qquad \epsilon_{(M,N)}\circ \eta_{(N,M)}=\id_{U_{{D}}}
\end{align*}
\end{defn}

\begin{defn}
If $M \in \mc{B}(C, D)$ and $N \in \mc{B}(D, C)$ define a Morita equivalence, then $C$ and $D$ are said to be \textbf{Morita equivalent}. 
\end{defn}

\begin{rmk}
Note that we use the dual pair as a subscript on the coevaluation and evaluation when there are multiple dual pairs.  We will use a similar notation for the Euler characteristic.  
\end{rmk}

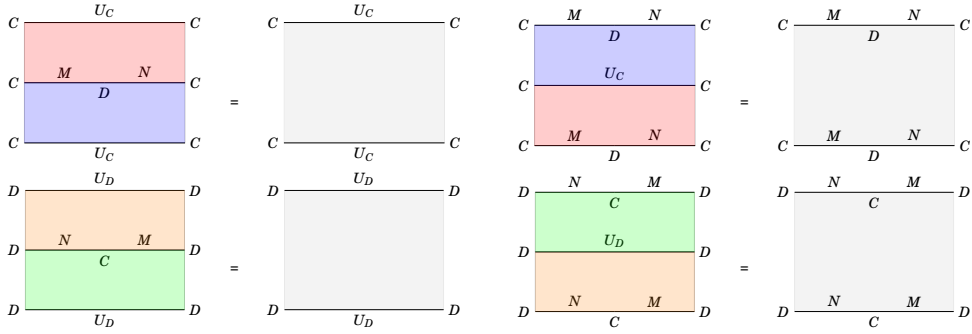
\begin{figure}
    \resizebox{.9\linewidth}{!}{
\begin{tikzpicture}
\coordinate (A1) at (0,3);
\node at (A1)  [left]{$C$};
\coordinate (A2) at (4,3);
\node at (A2) [right] {$C$};
\coordinate (A3) at (0,1.5);
\node at (A3)  [left]{$C$};
\coordinate (B1) at (2,1.5);
\node at (B1)  [below]{$D$};
\coordinate (A4) at (4,1.5);
\node at (A4) [right] {$C$};
\coordinate (A5) at (0,0);
\node at (A5)  [left]{$C$};
\coordinate (A6) at (4,0);
\node at (A6) [right] {$C$};
\draw (A1)--node[midway, above] {$U_C$}(A2);
\draw (A3)--node[midway, above] {$M$}(B1);
\draw (B1)--node[midway, above] {$N$}(A4);
\draw (A5)--node[midway, below] {$U_C$}(A6);
\draw[fill=red, opacity=0.2] (A1)--(A2)--(A4)--(A3)--(A1);
\draw[fill=blue, opacity=0.2] (A3)--(A4)--(A6)--(A5)--(A3);
\end{tikzpicture}
\hspace{.5cm}\raisebox{1.5cm}{=}\hspace{.5cm}
\begin{tikzpicture}
\coordinate (A1) at (0,3);
\node at (A1)  [left]{$C$};
\coordinate (A2) at (4,3);
\node at (A2) [right] {$C$};
\coordinate (A5) at (0,0);
\node at (A5)  [left]{$C$};
\coordinate (A6) at (4,0);
\node at (A6) [right] {$C$};
\draw (A1)--node[midway, above] {$U_C$}(A2);
\draw (A5)--node[midway, below] {$U_C$}(A6);
\draw[fill=black, opacity=0.05] (A1)--(A2)--(A6)--(A5)--(A1);
\end{tikzpicture}
\hspace{1cm}

\begin{tikzpicture}
\coordinate (A1) at (0,0);
\node at (A1)  [left]{$C$};
\coordinate (A2) at (4,0);
\node at (A2) [right] {$C$};
\coordinate (A3) at (0,1.5);
\node at (A3)  [left]{$C$};
\coordinate (B1) at (2,3);
\node at (B1)  [below]{$D$};
\coordinate (A4) at (4,1.5);
\node at (A4) [right] {$C$};
\coordinate (A5) at (0,3);
\node at (A5)  [left]{$C$};
\coordinate (A6) at (4,3);
\node at (A6) [right] {$C$};

\coordinate (B2) at (2,0);
\node at (B2)  [below]{$D$};
\draw (A3)--node[midway, above] {$U_C$}(A4);
\draw (A5)--node[midway, above] {$M$}(B1);
\draw (B1)--node[midway, above] {$N$}(A6);
\draw (A1)--node[midway, above] {$M$}(B2);
\draw (B2)--node[midway, above] {$N$}(A2);
\draw[fill=red, opacity=0.2] (A1)--(A2)--(A4)--(A3)--(A1);
\draw[fill=blue, opacity=0.2] (A3)--(A4)--(A6)--(A5)--(A3);
\end{tikzpicture}
\hspace{.5cm}\raisebox{1.5cm}{=}\hspace{.5cm}
\begin{tikzpicture}
\coordinate (A1) at (0,3);
\node at (A1)  [left]{$C$};
\coordinate (B1) at (2,3);
\node at (B1)  [below]{$D$};
\coordinate (A2) at (4,3);
\node at (A2) [right] {$C$};
\coordinate (A5) at (0,0);
\node at (A5)  [left]{$C$};
\coordinate (B2) at (2,0);
\node at (B2)  [below]{$D$};
\coordinate (A6) at (4,0);
\node at (A6) [right] {$C$};
\draw (A1)--node[midway, above] {$M$}(B1)--node[midway, above] {$N$}(A2);
\draw (A5)--node[midway, above] {$M$}(B2)--node[midway, above] {$N$}(A6);
\draw[fill=black, opacity=0.05] (A1)--(A2)--(A6)--(A5)--(A1);
\end{tikzpicture}
}

    \resizebox{.9\linewidth}{!}{
\begin{tikzpicture}
\coordinate (A1) at (0,3);
\node at (A1)  [left]{$D$};
\coordinate (A2) at (4,3);
\node at (A2) [right] {$D$};
\coordinate (A3) at (0,1.5);
\node at (A3)  [left]{$D$};
\coordinate (B1) at (2,1.5);
\node at (B1)  [below]{$C$};
\coordinate (A4) at (4,1.5);
\node at (A4) [right] {$D$};
\coordinate (A5) at (0,0);
\node at (A5)  [left]{$D$};
\coordinate (A6) at (4,0);
\node at (A6) [right] {$D$};
\draw (A1)--node[midway, above] {$U_D$}(A2);
\draw (A3)--node[midway, above] {$N$}(B1);
\draw (B1)--node[midway, above] {$M$}(A4);
\draw (A5)--node[midway, below] {$U_D$}(A6);
\draw[fill=orange, opacity=0.2] (A1)--(A2)--(A4)--(A3)--(A1);
\draw[fill=green, opacity=0.2] (A3)--(A4)--(A6)--(A5)--(A3);
\end{tikzpicture}
\hspace{.5cm}\raisebox{1.5cm}{=}\hspace{.5cm}
\begin{tikzpicture}
\coordinate (A1) at (0,3);
\node at (A1)  [left]{$D$};
\coordinate (A2) at (4,3);
\node at (A2) [right] {$D$};
\coordinate (A5) at (0,0);
\node at (A5)  [left]{$D$};
\coordinate (A6) at (4,0);
\node at (A6) [right] {$D$};
\draw (A1)--node[midway, above] {$U_D$}(A2);
\draw (A5)--node[midway, below] {$U_D$}(A6);
\draw[fill=black, opacity=0.05] (A1)--(A2)--(A6)--(A5)--(A1);
\end{tikzpicture}
\hspace{1cm}
\begin{tikzpicture}
\coordinate (A1) at (0,0);
\node at (A1)  [left]{$D$};
\coordinate (A2) at (4,0);
\node at (A2) [right] {$D$};
\coordinate (A3) at (0,1.5);
\node at (A3)  [left]{$D$};
\coordinate (B1) at (2,3);
\node at (B1)  [below]{$C$};
\coordinate (A4) at (4,1.5);
\node at (A4) [right] {$D$};
\coordinate (A5) at (0,3);
\node at (A5)  [left]{$D$};
\coordinate (A6) at (4,3);
\node at (A6) [right] {$D$};

\coordinate (B2) at (2,0);
\node at (B2)  [below]{$C$};
\draw (A3)--node[midway, above] {$U_D$}(A4);
\draw (A5)--node[midway, above] {$N$}(B1);
\draw (B1)--node[midway, above] {$M$}(A6);
\draw (A1)--node[midway, above] {$N$}(B2);
\draw (B2)--node[midway, above] {$M$}(A2);
\draw[fill=orange, opacity=0.2] (A1)--(A2)--(A4)--(A3)--(A1);
\draw[fill=green, opacity=0.2] (A3)--(A4)--(A6)--(A5)--(A3);
\end{tikzpicture}
\hspace{.5cm}\raisebox{1.5cm}{=}\hspace{.5cm}
\begin{tikzpicture}
\coordinate (A1) at (0,3);
\node at (A1)  [left]{$D$};
\coordinate (B1) at (2,3);
\node at (B1)  [below]{$C$};
\coordinate (A2) at (4,3);
\node at (A2) [right] {$D$};
\coordinate (A5) at (0,0);
\node at (A5)  [left]{$D$};
\coordinate (B2) at (2,0);
\node at (B2)  [below]{$C$};
\coordinate (A6) at (4,0);
\node at (A6) [right] {$D$};
\draw (A1)--node[midway, above] {$N$}(B1)--node[midway, above] {$M$}(A2);
\draw (A5)--node[midway, above] {$N$}(B2)--node[midway, above] {$M$}(A6);
\draw[fill=black, opacity=0.05] (A1)--(A2)--(A6)--(A5)--(A1);
\end{tikzpicture}
}
\caption{The diagrams for Morita equivalence.  For one dual pair the coevaluation is red and the evaluation is green.  For the other the coevaluation is orange and the evaluation is blue. }
\label{fig:morita}
\end{figure}

\begin{example}
In the bicategory of rings, bimodules, and homomorphisms, a Morita equivalence is a $(C,D)$-bimodule $M$ and a $(D,C)$-bimodule $N$ so that  $(M,N)$ and $(N,M)$ are dual pairs and, using the coevaluation and evaluation from these dual pairs, 
 \[
  M\otimes_DN\cong C\qquad\text{and} \qquad N\otimes_CM\cong D.
 \]
This is the familiar notion of Morita equivalence for rings and implies the  functor 
\[
 N\otimes_C- \colon \Mod_C\to \Mod_D
\] 
is equivalence of the category of (left) modules over $C$ and the category of (left) modules over $D$. 

The most familiar example of Morita equivalence is that the ring of $n$-by-$n$ matrices with elements in $C$ is Morita equivalent to $C$ for any $n > 0$. 
\end{example}

Morita equivalence is the correct notion of equivalence for 0-cells in a bicategory.  
As a result it should respect the Euler characteristic and the trace.  We now consider both of these, starting with the special case of the Euler characteristic.

\begin{prop}\label{euler_char_shadows}
Suppose  $(M,N)$ is a Morita equivalence. Then $\chi_{(M,N)}(M)=\chi_{(M,N)}(N)$ is an isomorphism with inverse $\chi_{(N,M)}(N)=\chi_{(N,M)}(M)$.  
\end{prop}

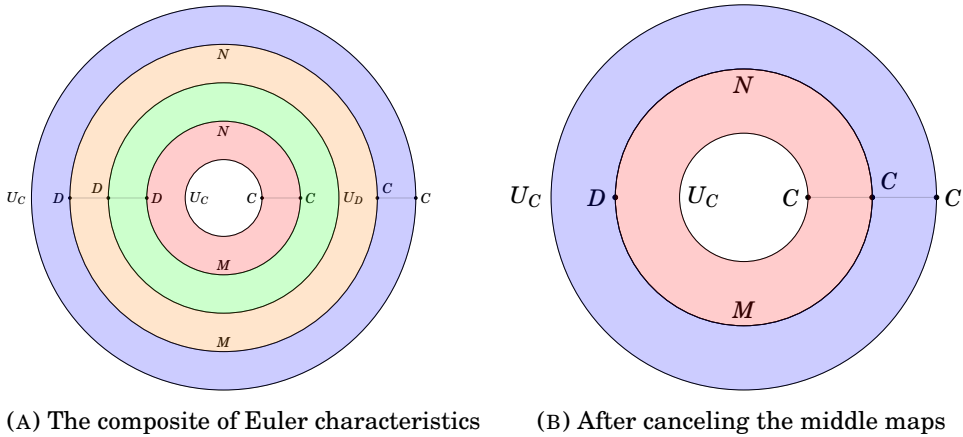
\begin{figure}
    \begin{subfigure}[b]{0.45\textwidth}
    \resizebox{.90\linewidth}{!}
{
\begin{tikzpicture}
\coordinate (A1) at (1,0);
\node at (A1) [left]{$C$};
\node (UA1) at (-1,0)[right]{$U_{C}$};

\draw (0,0) circle (1cm);
\filldraw (1,0) circle (1pt);

\draw (0,0) circle (2cm);

\filldraw (-2,0) circle (1pt);
\filldraw (2,0) circle (1pt);
\coordinate (A2) at (2,0);
\node at (A2) [right]{$C$};
\coordinate (B1) at (-2,0);
\node at (B1) [right]{$D$};
\node (M1) at (0,-2)[above]{$M$};
\node (N1) at (0,2)[below]{$N$};

\draw (0,0) circle (3cm);
\filldraw (-3,0) circle (1pt);
\coordinate (B2) at (-3,0);
\node at (B2) [above left ]{$D$};

\draw (0,0) circle (4cm);
\filldraw (-4,0) circle (1pt);
\coordinate (B3) at (-4,0);
\node at (B3) [left]{$D$};
\filldraw (4,0) circle (1pt);
\coordinate (A3) at (4,0);
\node at (A3) [above right]{$C$};
\node (M2) at (0,-4)[above]{$M$};
\node (N2) at (0,4)[below]{$N$};

\draw (0,0) circle (5cm);
\filldraw (5,0) circle (1pt);
\coordinate (A4) at (5,0);
\node at (A4) [right]{$C$};

\node (UB1) at (3,0)[right]{$U_{D}$};
\node (UA2) at (-5,0)[left]{$U_{C}$};

\draw [fill=red, opacity=0.2](A2)arc(0:360:2cm)--(A2) --(A1) arc(0:-360:1cm)--(A1);
\draw [fill=green, opacity=0.2](B2)arc(-180:-540:3cm)--(B2) --(B1) arc(180:540:2cm)--(B1);
\draw [fill=orange, opacity=0.2](B3)arc(-180:-540:4cm)--(B3) --(B2) arc(180:540:3cm)--(B2);
\draw [fill=blue, opacity=0.2](A4)arc(0:360:5cm)--(A4) --(A3) arc(0:-360:4cm)--(A3);
\end{tikzpicture}
}
        \caption{The composite of Euler characteristics}\label{fig:euler_char_compose}
\end{subfigure}
\begin{subfigure}[b]{0.45\textwidth}
    \resizebox{.98\linewidth}{!}
{
\begin{tikzpicture}
\coordinate (A1) at (1,0);
\node at (A1) [left]{$C$};
\node (UA1) at (-1,0)[right]{$U_{C}$};

\draw (0,0) circle (1cm);
\filldraw (1,0) circle (1pt);

\draw (0,0) circle (2cm);

\filldraw (-2,0) circle (1pt);
\filldraw (2,0) circle (1pt);
\coordinate (A2) at (2,0);
\coordinate (B1) at (-2,0);
\node (M1) at (0,-2)[above]{$M$};
\node (N1) at (0,2)[below]{$N$};

\draw (0,0) circle (2cm);
\filldraw (-2,0) circle (1pt);
\coordinate (B3) at (-2,0);
\node at (B3) [left]{$D$};
\filldraw (2,0) circle (1pt);
\coordinate (A3) at (2,0);
\node at (A3) [above right]{$C$};

\draw (0,0) circle (3cm);
\filldraw (3,0) circle (1pt);
\coordinate (A4) at (3,0);
\node at (A4) [right]{$C$};

\node (UA2) at (-3,0)[left]{$U_{C}$};

\draw [fill=red, opacity=0.2](A2)arc(0:360:2cm)--(A2) --(A1) arc(0:-360:1cm)--(A1);
\draw [fill=blue, opacity=0.2](A4)arc(0:360:3cm)--(A4) --(A3) arc(0:-360:2cm)--(A3);
\end{tikzpicture}
}
        \caption{After canceling the middle maps}\label{fig:euler_char_compose_2}
\end{subfigure}
        \caption{Euler characteristics and Morita equivalence (\cref{euler_char_shadows})}
\end{figure}

\begin{proof}  
\cref{fig:euler_char_compose} is a graphical representation of the composite of Euler characteristics where we have nested the red/green dual pair inside the  orange/blue dual pair. Following \cref{fig:morita} we can first cancel the concentric green and orange regions resulting in \cref{fig:euler_char_compose_2}.  Then the  red and blue regions cancel.  

the commutative diagrams below are a more formal proof. 
\[\xymatrix@C=40pt{
 \sh{U_C}\ar[r]^-{\sh{{\eta_{(M,N)}}}}\ar[rrrddd]^\id
& \sh{M\odot N}\ar[r]^\sim\ar[ddrr]^\id 
 &\sh{N\odot M}\ar[r]^-{\sh{\epsilon_{(M,N)}}} \ar[dr]^\id
 &\sh{U_D}\ar[d]^-{\sh{{\eta_{(N,M)}}}}
\\
\sh{U_D}\ar[d]_-{\sh{{\eta_{(N,M)}}}}\ar[ddrdrr]^\id 
 &&&\sh{N\odot M}\ar[d]^\sim
 \\
\sh{N\odot M}\ar[d]^\sim\ar[ddrr]^\id 
&&&\sh{M\odot N}\ar[d]^-{\sh{\epsilon_{(N,M)}}}
 \\
 \sh{M\odot N}\ar[d]_-{\sh{\epsilon_{(N,M)}}}\ar[dr]^\id
&&&\sh{U_C}
\\
\sh{U_C}\ar[r]_-{\sh{{\eta_{(M,N)}}}}
& \sh{M\odot N}\ar[r]^\sim
 &\sh{N\odot M}\ar[r]_-{\sh{\epsilon_{(M,N)}}} 
 &\sh{U_D}
}\]
\end{proof}

More generally, we have the following proposition. 

\begin{prop}\label{euler_char_shadows_2}
Suppose  $(M,N)$ is a Morita equivalence.
If $Q$ is a 1-cell so that $N\odot Q\odot M$ is defined, 
\[
 \sh{Q} \xto{\tr\left(\rsm{\eta}{M}{N}{Q}\right)} \sh{N\odot Q\odot M}
\]
is an  isomorphism with inverse $\sh{N\odot Q\odot M}\xto{\tr\left(\rsm{\epsilon}{N}{M}{Q}\right)} \sh{Q} $.
\end{prop} 

\begin{proof}  
Recall from \cref{rsm_simplified} that the  trace of $\rsm{\epsilon}{N}{M}{Q}$  is 
\[
 \sh{N\odot Q\odot M}\xto{\sim}\sh{Q\odot M\odot N}\xto{\sh{\id_Q\odot \epsilon_{(N,M)}}}\sh{Q\odot U_C}\cong \sh{Q}
\]
and the trace of $\rsm{\eta}MNQ$ is 
\[
 \sh{Q}\cong \sh{Q\odot U_C}\xto{\sh{\id_Q\odot \eta}}\sh{Q\odot M\odot N}\cong \sh{N\odot Q\odot M}.
\]

Composing these maps in both orders we have the following commutative diagrams.
\[\xymatrix{
 \sh{N\odot Q\odot M}\ar[r]^-{\theta}\ar[dddrrrr]^\id
 &\sh{Q\odot M\odot N}\ar[rr]^-{\sh{\id_Q\odot \epsilon_{(N,M)}}}\ar@/_10pt/[drr]_\id
 &&\sh{Q\odot U_C}\ar[r]\ar[d]_-{\sh{\id_Q\odot \eta_{(M,N)}}}\ar[dr]^\theta
 &\sh{Q}
 \\
 \sh{Q}\ar[dddrrrr]^\id
 &&&\sh{Q\odot M\odot N}\ar[dr]^\theta
 &\sh{U_C\odot Q}\ar[u]\ar[d]^-{\sh{\eta_{(M,N)}\odot \id_Q}}
 \\
 \sh{U_C\odot Q}\ar[d]_{\sh{\eta_{(M,N)}\odot \id_Q}}\ar[dr]^{\theta}\ar[u]
 &&&&\sh{M\odot N\odot Q}\ar[d]^-{\theta}
 \\
 \sh{M\odot N\odot Q}\ar[d]^-{\theta}\ar[dr]
 &\sh{Q\odot U_C}\ar@/^10pt/[drr]^\id\ar[d]^{\sh{\id_Q\odot \eta_{(M,N)}}}
 &&&\sh{N\odot Q\odot M}.
 \\
 \sh{N\odot Q\odot M}\ar[r]^-{\theta}
 &\sh{Q\odot M\odot N}\ar[rr]_-{\sh{\id_Q\odot \epsilon_{(N,M)}}}
 &&\sh{Q\odot U_C}\ar[r]
 &\sh{Q}
 }\]
\end{proof}

\begin{cor}\label{euler_char_shadows_4} 
Suppose  $(M_1,N_1)$ is a dual pair, $(M_2,N_2)$ is a Morita equivalence, $Q$ and $P$ are 1-cells so that 
$M_1\odot Q\odot N_1$ and $N_2\odot P\odot M_2$ are defined.  For a right dualizable 1-cell
$L$ so that $M_1\odot L\odot M_2$ is defined, and a 2-cell $f\colon Q\odot L\to L\odot P$ the following diagram commutes.
\[\xymatrix@C=50pt{
 \sh{M_1\odot Q\odot N_1}\ar[r]^-{\tr(M_1\odot f\odot M_2)}\ar[d]_-{\tr\left(\rsm{\epsilon}{M_1}{N_1}{Q} \right)}
 &\sh{N_2\odot P\odot M_2}\ar[d]^{\tr\left(\rsm{\epsilon}{N_2}{M_2}{P}\right) }
 \\
 \sh{Q}\ar[r]^-{\tr(f)}
 &\sh{P}
}\]

\end{cor}

\begin{proof}
This follows from \cref{euler_char_shadows_2,euler_char_shadows_3}. 
\end{proof}

\section{Euler characteristics for base change objects}\label{sect:euler_char}

Given a map $f\colon  A \to C$ of monoids in a symmetric monoidal category $\mc{V}$, we can define an $(A, C)$-bimodule $ _f C$ which has an action on the left through $f\colon  A \to C$. Similarly, we get a $(C, A)$-bimodule $C_f$. These objects are used to \textit{change base} in the following sense. Given a $(C,D)$-bimodule $M$, computing the composition $_f C\odot  M$ is the same as computing $(f\times \id_D)^\ast M$. 
These types of objects arise in any bicategory defined from an indexed monoidal category \cite{ps:indexed}, but we need not work in that generality here. Instead, we work in the two examples from \cref{sect:duality_and_trace}.

The main concern of \S\ref{sec:morita_base_change} will be recovering the classical ``Morita invariance'' statement that $\thh(A) \xrightarrow{\sim} \thh(\Mod^c_A)$ as well as a new, twisted, version of that statement by working in bicategories. Unraveling the necessary category theory, this hinges on the following question. Given a functor of $\mc{V}$-categories $F\colon  \mc{A} \to \mc{C}$, we have an associated map of 2-cells $F: U_{\mc{A}} \to U_{\mc{C}}$. Applying the shadow, we get a map $\sh{U_{\mc{A}}} \to \sh{U_{\mc{C}}}$. We could instead consider the 1-cell $_F \mc{C}$ (defined carefully in \cref{defn:base_change_category}) and compute $\chi (_F \mc{C}) \colon  \sh{U_{\mc{A}}} \to \sh{U_{\mc{C}}}$. It is clear that these maps should be the same, and we verify this in \cref{lem:base_change_euler_ex}. This simple observation is the core of what allows us to prove our main theorems. 

In \S\ref{sec:morita_trace} we use this identification to describe traces across Morita equivalences. 
Combining \cref{base_change_one_object,lem:base_change_euler_ex} gives 
\cref{ex:base_change_and_natural_transformation}.  This result is then used to show \cref{thm:lm_generalization}.

\subsection{Morita equivalence arising from base change}\label{sec:morita_base_change}
In what follows, let $\mc{V}$ be a symmetric monoidal category.

\begin{rmk}
We may  choose $\mc{V}$ to be any symmetric monoidal model category that satisfies the conditions of \cite[Prop.~6.1]{schwede_shipley}. In that case, the categories described below have associated model structures and homotopy categories. 
\end{rmk}

\begin{defn}\label{defn:base_change_monoid}
 Let $f\colon  A \to C$ be a morphism in $\mbf{Mon}(\mc{V})$. Then there is an $(A, C)$-bimodule  $ _f C$, which is $C$ with a left $A$-action given by $f$,  and a $(C, A)$-bimodule $C_f$ which is $C$ with a right $A$-action given by $f$.
\end{defn}

\begin{defn}\label{defn:base_change_category}
 Let $F\colon  \mc{A} \to \mc{C}$ be a morphism in $\mbf{Cat}(\mc{V})$, i.e. a functor of categories enriched in $\mc{V}$. Then there is an $(\mc{A}, \mc{C})$-bimodule $_F \mc{C}$ where the left action of $\mc{A}$ is given by
  \begin{align*}
  \mc{A}(a, a') \otimes ({_F \mc{C}})(a', c) 
   &= \mc{A}(a, a') \otimes \mc{C}(F(a'), c) 
  \\
  &\to \mc{C}(F(a), F(a')) \otimes \mc{C}(F(a'), c) \to \mc{C}(F(a), c) = {_F \mc{C}(a, c)} 
  \end{align*}
  There is a dual $(\mc{C}, \mc{A})$-bimodule that we denote $\mc{C}_F$. 
\end{defn}

\begin{defn}
We call any of $_f C$, $C_f$, $_F \mc{C}$, $\mc{C}_F$ \textbf{base change 1-cells}.
\end{defn}

\begin{rmk}
  It is important to remember that the maps $f\colon  A \to C$ and $F\colon  \mc{A} \to \mc{C}$ are \textit{not} 1-cells in the categories $\mc{B}(\mbf{Mon}(\mc{V}))$ and $\mc{B}(\mbf{Cat}(\mc{V}))$. They are the vertical 1-cells in an attendant double category (see, e.g. \cite{shulman_framed}).
\end{rmk}

In special cases, base change one cells may exhibit a Morita equivalence. We isolate this special case in a definition. 

\begin{defn}
A $\sV$-functor $F\colon \sA\to \sC$ 
is a {\bf Morita equivalence} if  $({_F\sC},\sC_F)$ is a Morita equivalence.  
In particular, 
 $F$ is a Morita equivalence if and only if $F$ is full and faithful and the map induced by composition  
\[
 \sC(c,F(-))\odot\sC(F(-),c')\to \sC(c,c')
\] 
is an isomorphism.
\end{defn}

The first step in answering the questions posed in the introduction to this section is to give descriptions of the coevaluation, evaluation, and Euler characteristic for base change objects.

\begin{prop}\cite[Appendix]{p:thesis}\cite[Lem.~7.6]{ps:indexed}\label{lem:base_change_euler}
\begin{enumerate}
  \item If $f\colon A\to C$ is a monoid homomorphism, $(_fC,C_f)$ is a dual pair.
  \item If $F\colon \sA\to \sC$ be a $\sV$-functor between $\sV$-categories, $({_F\sC},\sC_F)$ is a dual pair.
\end{enumerate}
\end{prop}

For objects $a$ and $a'$ in $\mc{A}$, a choice for the   
coevaluation is
\begin{align*}
 {\sA}(a,a')\xto{F} \sC(F(a),F(a'))&\cong \sC(F(a),F(a'))\otimes 1_{\sV}
 \\
 &\to \sC(F(a),F(a'))\otimes\sC(F(a'),F(a'))\to \sC(F(a),-)\odot\sC(-,F(a')) 
\end{align*}
For objects 
$c,c'$ of $\sC$, the corresponding  evaluation 
is induced by the composition of morphisms as in the following diagram.
\[\xymatrix{ 
  \coprod_{a\in \mc{A}} \sC(c,F(a))\otimes \sC(F(a),c') \ar[r]\ar[d]
  &\sC(c,c')
  \\
  \sC(c,F(-))\odot \sC(F(-),c') \ar@{.>}[ur]
}\]

If $\mc{M}$ is a $(\mc{C},\mc{D})$-bimodule and $F\colon \mc{A}\to \mc{C}$ and $G\colon \mc{B}\to \mc{D}$ are enriched functors, 
${_F\mc{M}_G}$ is the $(\mc{A},\mc{B})$-bimodule defined as the composite 
\[
  \mc{A}^{\text{op}}\otimes \mc{B}\xto{F\otimes G}\mc{C}^{\text{op}}\otimes \mc{D}\xto{\mc{M}}\mc{V}
\]
If $\mc{Q}$ is an $(\mc{C},\mc{C})$-bimodule and $F$ is as above, we have the following composite 
\[
  \rbm{\epsilon}{F}{\mc{C}}{\mc{Q}}\colon {_F\mc{Q}_F}\odot {_F\mc{C}}
  \cong {_F\mc{C}}\odot \mc{Q}\odot \mc{C}_F\odot {_F\mc{C}} 
  \xto{\rsm{\epsilon}{_F\mc{C}}{\mc{C}_F}{\mc{Q}}}{_F\mc{C}}\odot \mc{Q} 
\]
and the corresponding maps for $\rbm{\eta}{F}{\mc{C}}{\mc{Q}}$, $\lbm{\epsilon}{F}{\mc{C}}{\mc{Q}}$, and $\lbm{\eta}{F}{\mc{C}}{\mc{Q}}$.    The are also versions of these maps for monoids.

The following statement (and its restriction to the case of monoids) follows immediately from the 
coevaluation and evaluation above.
 
\begin{lem}\label{lem:base_change_euler_ex}
Let $F\colon \sA\to \sC$ be a $\sV$-functor between $\sV$-categories  and $\mc{Q}$ be a $(\mc{C},\mc{C})$-bimodule.  
Then the following two diagrams commute.
\[\xymatrix@C=12pt{
  \coprod_{a}\mc{A}(a,a)\ar[d]\ar[r]^-F&\coprod _{a}\mc{C}(F(a),F(a))\ar[r]
  &\coprod _c\mc{C}(c,c)\ar[d]
  \\ 
  \sh{U_\mc{A}}\ar[rr]^{\chi_{({_F\sC},\sC_F)}(_F\mc{C})=\chi_{({_F\sC},\sC_F)}(\mc{C}_F)}
  && \sh{U_\mc{C}}
 }
 \hspace{3em}
 \xymatrix{
  \coprod_{a} \mc{Q}(F(a),F(a))\ar[r]\ar[d]
  &\coprod_{c} \mc{Q}(c,c)\ar[d]
   \\
  \sh{{_F\mc{Q}_F}}\ar[r]^-{\tr\left(\rbm{\epsilon}{F}{\mc{D}}{\mc{Q}}\right)}
  &\sh{\mc{Q}}
}\]
\end{lem}
While unassuming and an immediate consequence of this choice of evaluation and coevaluation for base change dual pairs, this lemma is a fundamental connection between 
traces and maps of hom sets.  The left diagram above implies 
\[
  \chi_{({_F\sC},\sC_F)}(_F\mc{C})=\chi_{({_F\sC},\sC_F)}(\sC_F)=\sh{F}.
\] 
This observation will be used in \cref{cor:base_change_morita_isos,cor:lind_malkiewich,ex:base_change_and_natural_transformation}.  If $F$ is the inclusion of a subcategory, $\tr(\rbm{\epsilon}{F}{\mc{C}}{\mc{Q}})$ is the map on shadows induced by that 
inclusion.

\begin{cor}\label{cor:base_change_morita_isos}
If $F\colon \sA\to \sC$ is a Morita equivalence and $\mc{Q}$ is a $(\mc{C},\mc{C})$-bimodule 
\[
\sh{F} \colon \sh{\mc{A}}\to \sh{\mc{C}}
\quad\text{ and }\quad
\tr\left(\rbm{\epsilon}{F}{\mc{C}}{\mc{Q}}\right)\colon \sh{{_F\mc{Q}_F}}\rightarrow \sh{\mc{Q}}
\] 
are isomorphisms.
\end{cor}

\begin{proof}
\cref{euler_char_shadows,euler_char_shadows_2} with the substitutions 
\begin{center}
\begin{tabular}{c|c|c}
$M$
&$N$
&$Q$
\\
\hline
$_F\mc{C}$
&$\mc{C}_F$
&$\mc{Q}$
\end{tabular}
\end{center}
imply
\[
\chi_{({_F\sC},\sC_F)}(_F\mc{C})=\chi_{({_F\sC},\sC_F)}(\mc{C}_F)\colon \sh{\mc{A}}\to \sh{\mc{C}}
\quad\text{ and }\quad
\tr\left(\rbm{\epsilon}{F}{\mc{C}}{\mc{Q}}\right)\colon \sh{{_F\mc{Q}_F}}\rightarrow \sh{\mc{Q}}
\] 
are isomorphisms.  The remaining identification follows from \cref{lem:base_change_euler_ex}.
\end{proof}

The following is a crucial example. It is a classical fact for rings \cite[Prop.~2.1.5]{dundas_mccarthy}, and known for spectra \cite{blumberg_mandell} and it provides important motivation for this paper. The example shows that it follows from purely bicategorical facts.

\begin{example}\label{spectral_category_morita_equivalence}
  Let $\mc{B}(\mbf{Cat}(\operatorname{Sp}))$ be the bicategory of spectral categories (with the proper homotopy theoretic considerations, see \cref{bcat_example}) . Let $A$ be a ring spectrum. There are two spectral categories naturally associated with $A$
  \begin{enumerate}
  \item The spectral category $\operatorname{Mod}^c_A$. 
  \item The one object spectral category whose hom spectrum is $A$.  We denote this category $\End_{ \operatorname{Mod}^c_A}(A)$ since it is the full subcategory of  $\operatorname{Mod}^c_A$ with the single object $A$.
  \end{enumerate}
There is an inclusion functor 
\[
  E_A\colon  \End_{ \operatorname{Mod}^c_A}(A) \to \Mod^c_A
\] 
and so we can construct a 1-cell (i.e. a spectral bimodule) $ _{E_A} (\operatorname{Mod}^c_A)$. There is similarly a spectral bimodule  $ (\Mod^c_A)_{E_A}$. These two bimodules are a Morita equivalence.

Applying \cref{cor:base_change_morita_isos} with the substitutions
\begin{center}
  \begin{tabular}{c|c|c}
    $\mc{A}$ & $\mc{C}$ & $F$\\\hline
    $\End_{\operatorname{Mod}^c_A} (A)$ & $\Mod^c_A$ & $E_A$
  \end{tabular}
\end{center}
the map 
\begin{equation}\label{eq:thh_A_mod}
  \thh(A) \cong \thh\left(\End_{\Mod^c_A} (A)\right) \to \thh\left(\Mod^c_A\right)
\end{equation} 
induced by the inclusion of $A$ into $\Mod^c_A$ as a module over itself is an isomorphism.
We will give an explicit description of the inverse of this map on $\pi_0$ in \cref{lem:base_change_euler_inverse}.
\end{example}

\begin{example}\label{ex:twist_thh_morita}
Let $A$ be a ring spectrum,  $Q$ be an $(A,A)$-bimodule, and  
\[
-\sma_AQ\colon  \operatorname{Mod}^c_A \to \operatorname{Mod}^c_A
\] 
be given by $M \mapsto M \sma_A Q$.  Then $(\operatorname{Mod}^c_A)_{-\sma_AQ}$ is 
an  $(\operatorname{Mod}^c_A, \operatorname{Mod}^c_A)$-bimodule and the inclusion 
\[
  E_A\colon  \End_{ \operatorname{Mod}^c_A}(A) \to \operatorname{Mod}^c_A
\] 
defines a $(\End_{ \operatorname{Mod}^c_A}(A),\End_{ \operatorname{Mod}^c_A}(A))$-bimodule  $_{E_A}((\operatorname{Mod}^c_A)_{-\sma_AQ})_{E_A}$.  

Since $E_A$ is a Morita equivalence, applying \cref{cor:base_change_morita_isos} with the substitutions 
\begin{center}
\begin{tabular}{c|c|c|c}
$\mc{A}$
&$\mc{C}$
&$F$
&$\mc{Q}$
\\
\hline
$\End_{\operatorname{Mod}^c_A}(A)$
&$\operatorname{Mod}^c_A$
&$E_A$
&$(\operatorname{Mod}^c_A)_{-\sma_AQ}$
\end{tabular}
\end{center}
gives an isomorphism $\sh{_{E_A}((\operatorname{Mod}^c_A)_{-\sma_AQ})_{E_A}} \to \sh{(\operatorname{Mod}^c_A)_{-\sma_AQ}}$. We also have identifications
  \begin{align*}
    & \sh{ {_{E_A}( (\operatorname{Mod}^c_A)_{-\sma_AQ}})_{E_A}} = \sh{Q} = \thh(A; Q)\\
    &\sh{(\operatorname{Mod}^c_A)_{-\sma_AQ}} = \sh{(\operatorname{Mod}^c_A)_{-\sma_AQ}} = \thh(\operatorname{Mod}^c_A ; {-\sma_AQ}) 
  \end{align*}
  and so an isomorphism
  \[
  \thh(A; Q)\to \thh(\operatorname{Mod}^c_A; {-\sma_AQ})
  \]
  Thus, the classical Morita invariance of $\thh$ holds in a twisted context as well. 
\end{example}

\subsection{Morita equivalence and trace}\label{sec:morita_trace}
We now turn to the comparison of traces across Morita equivalence following \cref{euler_char_shadows,euler_char_shadows_2}.  We start with an example to help motivate the following results.
If $\phi\colon M\to M$ is a endomorphism of a right dualizable $(C,D)$-bimodule the trace of $\phi$ is a 
map 
\[
  \thh{(C)}\to \thh{(D)}.
\]  
We can also consider the functor $\Mod_C^c\to \Mod_D^c$ given by tensoring with $M$ on objects and tensoring with $\phi$ on morphisms.  With the Morita equivalences 
 for $C$ and $\Mod_C^c$ and for  $D$ and $\Mod_D^c$ these give us the following diagram.
\[\xymatrix@C=40pt{
  \thh{(C)}\ar[d]\ar[rr]^-{\tr(\phi)}
  &&\thh{(D)}\ar[d]
  \\
  \thh{(\Mod^c_C)}\ar[r]^-{\thh{( -\otimes\phi)}}
  &\thh{(\Mod^c_C;_{-\otimes_C M} (\Mod_D^c)_{-\otimes_C M} )}\ar[r]^-{\tr\left(\rbm{\epsilon}{-\otimes_CM}{\Mod^c_D}{\Mod^c_D}\right)}
  &\thh{(\Mod^c_D)}
}\] 
At this level of generality we can directly confirm this diagram commutes, but it will be more convenient to prove a significant generalization and then verify that this diagram is a special case.  This generalization (\cref{base_change_one_object}) is one of the main results of the paper and underlies the ideas in \cref{sect:transfer}.

We first fix some notation.

\begin{defn}
If $\mc{C}$ is a category enriched in $\mc{V}$ and $c$ is an object of $\mc{C}$ let $\End_\mc{C}(c)$ denote the  full subcategory of $\mc{C}$ whose single object is $c$.  
Let $E_c\colon \End_\mc{C}(c)\to \mc{C}$ be the inclusion.
\end{defn}

\begin{thm}\label{base_change_one_object}
Let $\sC$ and $\sD$ be  categories enriched over $\sV$, $c$ be an object of $\mc{C}$, $d$ be an object of $\mc{D}$,  $\mc{M}$ be a right dualizable $(\sC,\sD)$-bimodule, 
and suppose $E_d$ is a Morita equivalence.  
Let $\mc{Q}$ be a $(\sC,\sC)$-bimodule, $\mc{P}$ be a $(\mc{D},\mc{D})$-bimodule
and 
\[
  \phi\colon \mc{Q}\odot \mc{M}\to \mc{M}\odot \mc{P} 
\] 
be a natural transformation.  Then 
\begin{equation}\label{eq:base_change_one_object}
 \xymatrix@C=70pt{
  \sh{_{E_c}\mc{Q}_{E_c}}\ar[d]^{\tr\left(\rbm{\epsilon}{E_c}{\mc{C}}{\mc{Q}}\right)}\ar[r]^-{\tr({_{E_c}\mc{C}}\odot \phi\odot\mc{D}_{E_d})}
  &\sh{_{E_d}\mc{P}_{E_d}}\ar[d]^{\tr\left(\rbm{\epsilon}{E_d}{\mc{D}}{\mc{P}}\right)}
  \\
  \sh{\mc{Q}}\ar[r]^-{\tr(\phi)}
  &\sh{\mc{P}}}
\end{equation}
commutes. 
\end{thm}

\begin{proof}This is a consequence of 
 \cref{euler_char_shadows_4} with the substitutions 
\begin{center}
\begin{tabular}{c|c|c|c|c|c|c|c}
$M_1$
&$N_1$
&$M_2$
&$N_2$
&$Q$
&$P$
&$L$
&$f$
\\\hline
$_{E_c}\mc{C}$
&$\mc{C}_{E_c}$
&$\mc{D}_{E_d}$
&$_{E_d}\mc{D}$
&$\mc{Q}$
&$\mc{P}$
&$\mc{M}$
&$\phi$
\end{tabular}
\end{center}
\end{proof}

We can now recover a main result of Lind-Malkiewich \cite[Prop.~5.5]{lind_malkiewich}.

\begin{cor}\label{cor:lind_malkiewich}
Let $\sA$ and $\sC$ be  categories enriched over $\sV$, $a$ be an object of $\mc{A}$, $c$ be an object of $\mc{C}$, and  $F\colon \sA\to \sC$ be a $\mc{V}$-functor. 
If  $E_c\colon \End_\mc{C}(c)\to \mc{C}$ is a Morita equivalence, 
\[\xymatrix@C=65pt{
 \sh{\End_\mc{A}(a)}\ar[d]^{\sh{E_a}}\ar[r]^-{\chi({_{F\circ E_a}\sC_{E_c}})}
 &\sh{\End_\mc{C}(c)}\ar[d]^{\sh{E_c}}
 \\
 \sh{\mc{A}}\ar[r]^-{\sh{F}}&\sh{\mc{C}}
}\]
commutes.  
\end{cor}

Note that ${_{F\circ E_a}\sC_{E_c}}$ is $\mc{C}(F(a),c)$ is as a $(\End_\mc{A}(a), \End_\mc{C}(c))$-bimodule.

\begin{proof}
This follows from \cref{base_change_one_object} with the substitutions
\begin{center}
\begin{tabular}{c|c|c|c|c|c|c|c}
$\mc{C}$
&$\mc{D}$
&$c$
&$d$
&$\mc{M}$
&$\mc{Q}$
&$\mc{P}$
&$\phi$
\\
\hline
$\mc{A}$
&$\mc{C}$
&$a$
&$c$
&$_F\mc{C}$
&$U_\mc{A}$
&$U_\mc{C}$
&$\id$
\end{tabular}
\end{center}
\cref{lem:base_change_euler_ex} identifies the vertical and bottom maps.
\end{proof}

\begin{example}\label{morita_invariance}
Let $C$ and $D$ be  rings and $M$ be a $(C,D)$-bimodule that is  that is finitely generated and projective as an right $D$-module. 

By \cref{cor:lind_malkiewich} with the substitutions 
\begin{center}
\begin{tabular}{c|c|c|c|c}
$\mc{A}$
&$a$
&$\mc{C}$
&$c$
&$F$
\\
\hline
$\Mod_C^c$
&$C$
&$\Mod_D^c$
&$D$
&$-\otimes_CM$
\end{tabular}
\end{center}
following diagram, where the vertical maps are inclusions,  commutes.
\[\xymatrix@C=40pt{
  \sh{{C}}\ar[r]^-{\chi(M)}\ar[d]
  &\sh{D}\ar[d]
  \\
   {\sh{{\operatorname{Mod}^c_C}}}\ar[r]^-{\sh{-\otimes_C M  }}
  &{\sh{{\operatorname{Mod}^c_D}}}
} \]
\end{example}

In later examples the map $\phi$ in \cref{base_change_one_object} is a composite 
\begin{equation}\label{switch_for_trace}
 \mc{Q}\odot \mc{M}\xto{\psi \odot  \id_\mc{M}}\mc{M}\odot  \mc{P}\odot \mc{N} \odot \mc{M}\xto{\rsm{\epsilon}{\mc{M}}{\mc{N}}{\mc{P}}} \mc{M}\odot \mc{P}
\end{equation}
for a map $\psi\colon  \mc{Q}\to \mc{M}\odot \mc{P} \odot \mc{N}$.
Then  \cref{prop:tightening} implies the trace of \eqref{switch_for_trace} is the composite 
$\sh{\mc{Q}}\xto{\sh{\psi}}\sh{\mc{M}\odot \mc{P} \odot \mc{N}}\xto{\tr\left(\rsm{\epsilon}{\mc{M}}{\mc{N}}{\mc{P}}\right)}\sh{\mc{P}}$
and \eqref{eq:base_change_one_object} becomes
\begin{equation}\label{switch_for_trace_diagram}
\xymatrix{
 \sh{{_{E_c}\mc{Q}_{E_c}}}\ar[d]^{\tr\left(\rbm{\epsilon}{E_c}{\mc{C}}{\mc{Q}}\right)}
   \ar[rrr]^-{\tr\left({_{E_c}\mc{C}}\odot \phi\odot \mc{D}_{E_d}\right)}
 &&&\sh{_{E_d}\mc{P}_{E_d}}\ar[d]^{\tr\left(\rbm{\epsilon}{E_d}{\mc{D}}{\mc{P}}\right)}
 \\
 \sh{\mc{Q}}\ar[r]^-{\sh{\psi}}
 &\sh{\mc{M}\odot \mc{P}\odot \mc{N}}\ar[rr]^-{\tr\left(\rsm{\epsilon}{\mc{M}}{\mc{N}}{\mc{P}}\right )} 
 &&\sh{\mc{P}}
 }
\end{equation}

An important example of this is the 2-cells that arise from a natural transformation as in \cref{ex:base_change_and_natural_transformation}.  We first describe how these two cells are defined.
For enriched categories  $\mc{A}$ and $\mc{C}$,  enriched functors
\[
J\colon \mc{A}\to\mc{A}, K\colon \mc{C}\to \mc{C}\text{ and }F\colon \mc{A}\to \mc{C},
\]
 and a natural transformation $\Phi\colon F\circ J\to K\circ F$
let
$\phi\colon \mc{A}_{J}\odot {_F\mc{C}} \to {_F\mc{C}\odot\mc{C}_K }$ be the 2-cell defined by 
\begin{align*}
 (a\xto{\alpha}J(a'),F(a')\xto{\beta}c)&\mapsto (F(a)\xto{F(\alpha)}F(J(a'))\xto{ \Phi_{a'}}K(F(a'))\xto{K(\beta)}K(c))
\end{align*} 
and $\psi\colon \mc{A}_J\to{_F\mc{C}\odot \mc{C}_K\odot \mc{C}_F} $ be 2-cell defined by 
\[
 (a\xto{\alpha} J(a'))\mapsto  (F(a)\xto{F(\alpha)}F(J(a'))\xto{\Phi_{a'}}K(F(a')))
\]
This choice of $\phi$ and $\psi$ are related as in \eqref{switch_for_trace}.

\begin{cor}\label{ex:base_change_and_natural_transformation}
For $\mc{A}$, $\mc{C}$, $J$, $K$, $F$, $\phi$ and $\psi$ as above and objects $a$ of $\mc{A}$ and $c$ of $\mc{C}$, 
the following diagrams, where  vertical maps are induced by inclusions on hom sets, commute.
\[\xymatrix@C=60pt{
  \sh{\mc{A}(a,J(a))} \ar[d] \ar[r]^-{\tr\left({_{E_a}\mc{A}}\odot \phi\odot \mc{C}_{E_{c}}\right)}
  &\sh{\mc{C}(c,K(c))}\ar[d]
  \\
  \sh{\mc{A}_J}\ar[r]^-{\tr(\phi)}
  &\sh{\mc{C}_K}
  }
\hspace{10pt}
\xymatrix@C=5pt{
  \sh{\mc{A}(a,J(a))} \ar[d] \ar[rr]^-{\tr\left({_{E_a}\mc{A}}\odot \phi\odot \mc{C}_{E_{c}}\right)}
  &&\sh{\mc{C}(c,K(c))}\ar[d]
  \\
  \sh{\mc{A}_J}\ar[r]^-{\sh{\psi}}&\sh{_F\mc{C}\odot \mc{C}_K\odot \mc{C}_F}\ar[r] 
  &\sh{\mc{C}_K}
}\]
The remaining unlabeled map is induced by the map 
\[
  \coprod_{a\in \mc{A}}\mc{C}(F(a), K(F(a)))\to \coprod_{c\in \mc{C}}\mc{C}(c, K(c)).
\] 
\end{cor}

\begin{proof} 
\cref{base_change_one_object} with the substitutions
\begin{center}
\begin{tabular}{c|c|c|c|c|c|c}
$\mc{C}$
&$c$
&$\mc{D}$
&$d$
&$\mc{M}$
&$\mc{Q}$
&$\mc{P}$
\\
\hline
$\mc{A}$
&$a$
&$\mc{C}$
&$c$
&$_F\mc{C}$
&$\mc{A}_J$
&$\mc{C}_K$
\end{tabular}
\end{center}
 gives the following commutative diagram.
\[\xymatrix@C=70pt{
 \sh{{_{E_a}(\mc{A}_J)_{E_a}}}
  \ar[d]^{\tr\left(\rbm{\epsilon}{E_a}{\mc{A}}{\mc{A}_J}\right)}
  \ar[r]^-{\tr\left({_{E_a}\mc{A}}\odot \phi\odot \mc{C}_{E_{c}}\right)}
 &\sh{_{E_{c}}(\mc{C}_K)_{E_{c}}}\ar[d]^{\tr\left(\rbm{\epsilon}{E_{c}}{\mc{C}}{\mc{C}_K}\right)}
 \\
 \sh{\mc{A}_J}\ar[r]^-{\tr(\phi)}
 &\sh{\mc{C}_K}
}\]
The diagram in \eqref{switch_for_trace_diagram} becomes
\[\xymatrix{
 \sh{{_{E_a}(\mc{A}_J)_{E_a}}}
  \ar[d]^{\tr\left(\rbm{\epsilon}{E_a}{\mc{A}}{\mc{A}_J}\right)}
  \ar[rrr]^-{\tr\left({_{E_a}\mc{A}}\odot \phi\odot \mc{C}_{E_{c}}\right)}
 &&&\sh{_{E_{c}}(\mc{C}_K)_{E_{c}}}\ar[d]^{\tr\left(\rbm{\epsilon}{E_{c}}{\mc{C}}{\mc{C}_K}\right)}
 \\
 \sh{\mc{A}_J}\ar[r]^-{\sh{\psi}}
 &\sh{_F\mc{C}\odot \mc{C}_K\odot \mc{C}_F}
  \ar[rr]^-{\tr\left(\rbm{\epsilon}{F}{\mc{C}_F}{\mc{C}_K}\right )} 
 &&\sh{\mc{C}_K}
}\]
The remaining simplifications follow from  \cref{lem:base_change_euler_ex}.
\end{proof}

For later applications it is convenient to note that ${_{E_a}\mc{A}}\odot \phi\odot \mc{C}_{E_c}$ is 
\begin{align*}
 \mc{A}(a,J(a))\odot \mc{C}(F(a),c) &\to \mc{C}(F(a),K(c))
 \\
 (a\xto{\alpha}J(a),F(a)\xto{\beta}c)&\mapsto (F(a)\xto{F(\alpha)}F(J(a))\xto{ \Phi_a}K(F(a))\xto{K(\beta)}K(c))
\end{align*} 

\begin{example} 
We now return to the example at the beginning of this subsection.  Let $C$ and $D$ be rings, $Q$ be an $(C,C)$-bimodule, $P$ be an $(D,D)$-bimodule, $M$ be an $(C,D)$-bimodule that is finitely generated and projective as an right $D$-module, and let $f\colon Q\otimes_C M\to M\otimes_D P$ be a homomorphism. 

\cref{ex:base_change_and_natural_transformation} with the  substitutions 
\begin{center}
\begin{tabular}{c|c|c|c|c|c|c|c}
$\mc{A}$
&$a$
&$\mc{C}$
&$c$
&$J$
&$K$
&$F$
&$\Phi$
\\
\hline
$\Mod_C^c$
&$C$
&$\Mod_D^c$
&$D$
&${-\otimes_CQ}$
&${-\otimes_DP}$
&${-\otimes_CM}$
&$-\otimes_Cf$
\end{tabular}
\end{center}
implies the bottom square in the following diagram commutes.
\[\xymatrix{
  \sh{Q}	\ar[d]	\ar[rr]^-{\tr(f)}
  &&\sh{P}\ar[d]
  \\
  \sh{\Mod_C^c(C,Q)}
	\ar[d]
	\ar[rr]^-{\tr(f_*)}
  &&\sh{\Mod_D^c(D,P)}\ar[d]
  \\
  \sh{(\Mod_C^c)_{-\otimes_CQ}}\ar[r]^-{\sh{\psi}}
  &\sh{_{-\otimes_CM}(\Mod_D^c)_{-\otimes_CM\otimes_DP}}\ar[r]
  &\sh{(\Mod_D^c)_{-\otimes_DP}}
}\]
 In this example, 
${_{E_C}(\Mod_C^c)}\odot \phi\odot (\Mod_D^c)_{E_D}$ is 
\begin{align*}
 \Mod_C^c(C, Q)\odot \Mod_D^c(M,D) &\to \Mod_D^c(M, P)
 \\
 (C\xto{\alpha} Q,M\xto{\beta}D)
 &\mapsto (C\otimes_CM\xto{\alpha\otimes\id}Q\otimes_AM\xto{ f}M\otimes _DP\xto{\beta\otimes \id}D\otimes_DP)
\end{align*} 
The top square is the identification of $\Mod_C^c(C,Q)$ with $Q$ and $\Mod_D^c(D,P)$ with $P$ and the observation that the trace of the dual of a map is the trace of the map.
\end{example}

\section{Example: Fixed Point Invariants}\label{sect:transfer}
We now return to the motivating example from fixed point theory described in the introduction and prove  \cref{thm:lm_generalization} in \cref{thm:main_theorem}.

Let $R$ be a ring spectrum and $A$, $C$ be $R$-algebras.  A map of $R$-algebras $f\colon  A \to C$ defines an $(A, C)$-bimodule $_fC$.  If 
$_fC$
 is left dualizable, we have 
an adjunction
\[
 C_f \sma_A - \colon  \operatorname{Mod}^{lc}_{(A, R)} \leftrightarrows \operatorname{Mod}^{lc}_{(C, R)} \colon  \ _f C \sma_C - 
\]
between the bimodules that are left-compact (that is, for example, an $(A, R)$-bimodule is compact when considered as an $A$-module). Letting $R$ be the sphere spectrum and taking $\thh$ of both sides of the adjunction gives us maps
\[
\operatorname{res}\colon  \thh(\operatorname{Mod}^c_{A}) \leftrightarrows \thh(\operatorname{Mod}^c_C) \colon  \operatorname{trf} 
\]
which are usually referred to as \textbf{restriction} and \textbf{transfer}. There is a restriction for any map $f$, but there is only a transfer if 
$C$ is compact as an $A$-module. Using Morita invariance 
restriction and transfer  give maps between $\thh(A)$ and $\thh(C)$. 
We should note that the transfer map has appeared in many $\thh$ calculations and seems to provide powerful characteristic-type invariants \cite{schlichtkrull_1, schlichtkrull_2, bentzen_madsen}.

It is well know that the composite
\[
S \cong \thh(S) \xrightarrow{\operatorname{trf}} \thh(\Sigma^\infty_+ \Omega X) \xto{\operatorname{res} }  \thh (S) \simeq S 
\]
is the Euler characteristic.  In this section we show that similar results hold for the generalizations of the Euler characteristics used in fixed point theory.
In particular we show that the spectrum-level Reidemeister trace of the second author \cite{p:thesis,p:coincidence} arises from transfer maps in $\thh$. The transfer maps we use are ``twisted'' by a bimodule coordinate in $\thh$. 

Since transfer maps are nothing more than an example of base change we can apply the results of the previous section.
Taking a more bicategorical perspective,
we have a map
\begin{equation}\label{eq:base_change_trace_1}
U_C \odot   C_f \to   C_f \odot U_A  
\end{equation}
which, upon taking traces gives us a map $\sh{U_C} \to \sh{U_A}$, or $\thh(C) \to \thh(A)$. This is an example of what is referred to as an Euler characteristic above. Similarly,
\begin{equation}\label{eq:base_change_trace_2}
U_A \odot {_f C}\to {_f C} \odot U_C
\end{equation}
gives $\sh{U_A} \to \sh{U_C}$, i.e. $\thh(A) \to \thh(C)$. 
Then the following result  is simply  a restatement of \cref{morita_invariance}.

\begin{prop}\label{prop:lm_final_statement}
  The diagrams
  \[
  \xymatrix{
    \thh(C) \ar[d]^\sim \ar[r]^-{\tr(\ref{eq:base_change_trace_1})} & \thh (A) \ar[d]^\sim 
\\
    \thh(\operatorname{Mod}^c_C) \ar[r]^-{\operatorname{trf}} & \thh(\operatorname{Mod}^c_A)
  }
  \qquad
    \xymatrix{
    \thh(A) \ar[d]^\sim \ar[r]^-{\tr(\ref{eq:base_change_trace_2})} & \thh (C) \ar[d]^\sim 
\\	
    \thh(\operatorname{Mod}^c_A) \ar[r]^-{\operatorname{res}}  & \thh(\operatorname{Mod}^c_C)
  }
  \]
commute.
\end{prop}

\begin{rmk}
The above was proved in \cite{lind_malkiewich}, as a step in verifying the first-named author's conjecture that the Becker-Gottlieb transfer factors through the $\thh$ transfer. 
\end{rmk}

Thus, 
restriction and transfer maps are  examples of a rather trivial bicategorical trace: the Euler characteristic. It is no coincidence that transfer maps in $\thh$ seem to produce results reminiscent of characteristics. The proposition further gives a very small, relatively computable model for the transfer map. In fact, it gives compact formulas for the usual restriction-transfer compositions

\begin{thm}[Restriction-Transfer]
  \begin{align*}
    \operatorname{res} \circ \operatorname{trf} = \chi(_fC_f) \colon  \thh(A) \to \thh(A)\\
    \operatorname{trf} \circ \operatorname{res} =  \chi((C_f)\odot (_fC))  \colon  \thh(C) \to \thh(C). 
  \end{align*}
\end{thm}

\begin{example}
  We present two examples of these types of restriction-transfer identities
\begin{itemize}
\item Let $X$ be a path-connected space of  the homotopy type of a finite CW-complex. If we let $A=S$ and $C=\Sigma^\infty_+ \Omega X$, then $A$ is compact as a $C$-module.  The composite 
  \[
  S \simeq \thh (S) \xto{ \operatorname{trf}}  \thh (\Sigma^\infty_+ \Omega X) \xto{\operatorname{res} } \thh (S) \simeq S 
  \]
is $\chi(X)$, regarded as an element of $\pi_0(S)$.
\item Let $H$ be a subgroup of $G$ where $[G:H] < \infty$.  Then the base change object associated to the inclusion of the group ring $A = \Z[H]$ into $C = \Z[G]$ is dualizable. The shadow is $\operatorname{HH}_0 (\Z[G], \Z[G])$ and 
  \[
  \operatorname{HH}_0 (\Z[H],\Z[H]) \xto{ \operatorname{trf}}   \operatorname{HH}_0 (\Z[G],\Z[G]) \xto{\operatorname{res} } \operatorname{HH}_0 (\Z[H], \Z[H]) 
  \]
  is multiplication by $[G:H]$. This has an interpretation in terms of the character theory of group representations, since $\operatorname{HH}_0 (\Z[G], \Z[G]) \cong \Z[C_G]$ where $C_G$ denotes the conjugacy classes of $G$. Since maps from $\Z[C_G]$ are exactly class functions, this is the same $[G:H]$ that appears in the classical induction-restriction maps in character theory for representations of finite groups.
\end{itemize}
\end{example}

To extend these ideas to the Reidemeister trace we need an elaboration of the transfer.
For algebras $A$ and $C$, a commuting diagram of maps of algebra
\[\xymatrix{A\ar[r]^j\ar[d]^f&A\ar[d]^f\\C\ar[r]^k&C}\]
defines a commuting diagram of functors
\[\xymatrix{
\Mod_C^c\ar[d]_-{-\sma_C C_f}\ar[r]^-{-\sma_C C_k}&\Mod_C^c\ar[d]^-{-\sma_C C_f}
\\
\Mod_A^c\ar[r]^-{-\sma_A A_j}&\Mod_A^c}
\] and 
 induces maps 
\begin{gather}\label{eq:main_base_change_map} C_k\odot C_f\to C_f\odot A_j.
\\
\label{eq:thh_twist_map}
(\Mod_C^c)_{-\wedge_CC_k}\odot {_{-\wedge_CC_f}(\Mod_A^c)}\to 
{_{-\wedge_CC_f}(\Mod_A^c)}\odot (\Mod_A^c)_{-\wedge_AA_j}
\end{gather} 

\begin{prop}\label{prop:main_prop} Suppose $f$, $j$, and $k$ are as above.   If $C_f$ is right dualizable then 
  the following diagram commutes. 
  \[
  \xymatrix@R=.5cm@C=1.5cm{
    \thh(C;C_k) \ar[r]^-{\tr\eqref{eq:main_base_change_map}}  \ar[d]^\sim& \thh(A; A_j)\ar[d]^\sim\\
    \thh(\operatorname{Mod}^c_C;- \wedge C_k) \ar[r]^-{\tr\eqref{eq:thh_twist_map}}  
  & \thh(\operatorname{Mod}^c_A ; -\wedge A_j).
    }
  \]
\end{prop} 

\begin{proof}
This follows from \cref{ex:base_change_and_natural_transformation} with the substitutions
\begin{equation*}
\begin{tabular}{c|c|c|c|c|c}
$\mc{A}$
&$\mc{C}$
&$I$
&$K$
&$F$
&$\Phi$
\\
\hline
$\Mod_C^c$
&$\Mod_A^c$
&$-\sma_C C_k$
&$-\sma_A A_j$
&$-\sma_C C_f$
&$\id$
\end{tabular}
\end{equation*}
\end{proof}

If $k$ is the identity map this reduced to a commutative diagram
  \[
  \xymatrix@R=.5cm@C=1.5cm{
    \thh(C) \ar[r]^-{\tr\eqref{eq:main_base_change_map}}  \ar[d]^\sim
    & \thh(A; A_j)\ar[d]^\sim
   \\
    \thh(\operatorname{Mod}^c_C) \ar[r]^-{\tr\eqref{eq:thh_twist_map}}  
  & \thh(\operatorname{Mod}^c_A ; -\wedge A_j).
    }
  \]
In this case the bottom map is the composite 
\[
  \thh(\operatorname{Mod}^c_C) \to\thh(\operatorname{Mod}^c_A) \to
  \thh(\operatorname{Mod}^c_A ; -\wedge A_j)
\]
where the first map is the transfer and the second map is induced by the map $S\to A_j$.

We now come to one of our main applications, which is simply a particular case of \cref{prop:main_prop} (and so a consequence of \cref{base_change_one_object}). Let $X$ be path-connected space that has the homotopy type of a finite CW-complex, and let $f\colon  X \to X$ be a self map. The map $X \to \ast$ gives a map $\Sigma^\infty_+ \Omega X \to S$ which exhibits $S$ as a compact $\Sigma^\infty_+ \Omega X$-module, and so there is a transfer $S \to \thh(\Sigma^\infty_+ \Omega X)$. Also, $f$ induces a self-map $f\colon  \Sigma^\infty_+ \Omega X \to \Sigma^\infty_+ \Omega X$. 

\begin{thm}\label{thm:main_theorem}
  The Reidemeister trace of $f\colon  X \to X$ is  the $\thh$ transfer
  \[
  S \simeq \thh(S) \to \thh(\Sigma^\infty_+ \Omega X;-\sma \Sigma^\infty_+ \Omega X_f) \simeq \Sigma^\infty_+ \mc{L} X^f 
  \]
\end{thm}
Recall that twisted $\thh$ was defined in \cref{right_twisted}, and $\mc{L}X^f$ in the introduction --- but see also \cref{defn:twisted_loop_space}. 

\begin{proof}
In \cite[Prop.~3.2.3]{p:thesis} the Reidemeister trace of the introduction is identified with a bicategorical trace.  This trace is identified with a more homotopically satisfying trace in 
\cite[Prop.~6.2.2]{p:thesis} and finally identified with a trace in the parameterized stable homotopy bicategory,  $\operatorname{Ex}$, in \cite[Thm.~4.1]{p:coincidence}.  
  
In the bicategory of parameterized spectra, let $S_X$ denote the fiberwise suspension spectrum of $X$ regarded as a space over $\ast\times X$ via the identity map.  If $f\colon X\to X$ is a continuous map $X_f$ is the  fiberwise suspension spectrum of  $\{(\gamma,x)\in X^I\times X|\gamma(1)=f(x)\}$
regarded as a space over $X\times X$ via the map $(\gamma,x)\mapsto (\gamma(0),x)$.
Then $f$ defines a map of spectra
\begin{equation}\label{twisted_map_of_X_f} 
S \odot S_X \to S_X \odot X_f
\end{equation}
If $X$ is a closed smooth manifold, $S_X$ is right dualizable and the trace of \eqref{twisted_map_of_X_f} 
is a map $\sh{S}\to \sh{X_f}$.  This map is one form of the Reidemeister trace.

Now, we use the Morita equivalence of \cite{lind_malkiewich_morita} to compare the shadows in $\operatorname{Ex}(X, X)$ and $(\Sigma^\infty_+ \Omega X, \Sigma^\infty_+ \Omega X)$-bimodules. Let $\mc{B}\text{imod}_S$ be the bicategory of ring spectra and spectral bimodules, and $\mc{B}\text{imod}^{\text{gp}}_S$ be the full-subcategory of ring spectra of the form $\Sigma^\infty_+ \Omega X$.  Then, \cite[Thm.~6.4]{lind_malkiewich_morita} as applied in \cite[Thm.~5.1]{lind_malkiewich} gives that the parameterized stable homotopy bicategory (when restricted to connected spaces) and $\mc{B}\text{imod}^{\text{gp}}_S$ are bicategorically equivalent with an induced equivalence on shadows. Under this equivalence (given in detail before \cite[Thm.~5.1]{lind_malkiewich}) the map of parameterized spectra in \eqref{twisted_map_of_X_f}  passes to the map of $(S, \Sigma^\infty_+ \Omega X)$-bimodule spectra 
\begin{equation}\label{Reidemeister_map_of_modules}
S \sma_{S} \Sigma^\infty_+ \Omega X \to \Sigma^\infty_+\Omega X \sma_{\Sigma^\infty_+ \Omega X} \Sigma^\infty_+ \Omega X_f.
\end{equation}
The functoriality of the trace \cite{p:thesis} implies the following diagram  commutes.
\[
\xymatrix@C=40pt{
  \sh{S} \ar[d]^{\sim}\ar[r]^{\tr\eqref{twisted_map_of_X_f} } & \sh{X_f}\ar[d]^{\sim} \\
  \sh{S}\ar[r]^-{\tr\eqref{Reidemeister_map_of_modules}} & \sh {\Sigma^\infty_+ \Omega X_f}.
}
\]

Then \cref{prop:main_prop} implies
\[
\xymatrix{
  \sh{S} \ar[d]^{\sim}\ar[r]^-{\tr\eqref{Reidemeister_map_of_modules}}  & \sh {\Sigma^\infty_+ \Omega X_f}\ar[d]^{\sim}\\
  \thh(\operatorname{Mod}^c_S) \ar[r] & \thh(\operatorname{Mod}^c_{\Sigma^\infty_+ \Omega X}; -\wedge \Sigma^\infty_+ \Omega X_f)
}
\]
commutes and 
\cref{cor:identify_thh_loop} identifies \[\thh(\Sigma^\infty_+\Omega X; \Sigma^\infty_+ \Omega X_f)\simeq \Sigma^\infty_+ \mc{L} X^f .\qedhere\]
\end{proof}

\section{A $\pi_0$-level cyclotomic trace}\label{sect:cyclotomic}

Much of the motivation for the present paper comes from very concrete questions about $\thh$. Above we used bicategories to understand the relationship between base change and traces using shadows, which yields information about $\thh$. Here we consider a different, though related, question that will be essential for future work.  Let $A$ be a ring spectrum, and let $f\colon  P \to P$ be an endomorphism of an $A$-module $P$. Then $f$ determines an element of $\pi_0 (\operatorname{THH}(\operatorname{Mod}^c_A))$ as follows. First, it determines a map of modules $S \to \operatorname{End}(P) $ adjoint to $f$. This includes into the zero skeleton of $\thh(\operatorname{Mod}^c_A)$, which is
\[
\operatorname{sk}^0 (\thh(\operatorname{Mod}^c_A)) = \bigvee_{M \in \operatorname{Mod}^c_A} \operatorname{End}(M) = \bigvee_{M \in \operatorname{Mod}^c_A} \operatorname{Mod}^c_A (M, M). 
\]
The zero skeleton includes into $\operatorname{THH}(\operatorname{Mod}^c_A)$. Thus, we have a composite
\begin{equation*}
S \to \operatorname{End}_{\Mod^c_A}(P) \hookrightarrow \bigvee_{M \in\operatorname{Mod}^c_A} \operatorname{Mod}^c_A (M, M) \hookrightarrow \thh(\operatorname{Mod}^c_A)
\end{equation*}
Finally, Morita invariance extends this map to
\begin{equation}\label{easy_trace}
S \to \operatorname{End}_{\Mod^c_A}(P) \hookrightarrow \bigvee_{M \in\operatorname{Mod}^c_A} \operatorname{Mod}^c_A (M, M) \hookrightarrow \thh(\operatorname{Mod}^c_A) \xrightarrow{\sim} \thh(A)
\end{equation}
yielding a map $S \to \operatorname{THH}(A)$, i.e. an element of $\pi_0 \operatorname{THH}(A)$. The question is:  What element of $\pi_0 \operatorname{THH}(A)$ does $f$ determine? Intuitively, the answer should clearly be $\operatorname{tr}(f)$, but it is far from obvious given how we have defined it. It is therefore desirable to have an actual proof of this, and a categorical proof is even better.  This is given in \cref{prop:base_change_euler_monoid}.

\begin{rmk}
  Since traces are additive, the association $[f] \mapsto \operatorname{tr}(f)$ given by the composite \eqref{easy_trace} descends to a map from the Grothendieck group of endomorphisms of $A$-modules, also called $K_0 (\operatorname{End}(A))$. Thus, we obtain a map
  \[
  K_0 (\operatorname{End}(A)) \to \pi_0 \operatorname{THH}(A).
  \]
  Indeed, connoiseurs will recognize \eqref{easy_trace} as the usual ``inclusion of objects'' map that is used to define the cyclotomic trace (see, e.g. \cite{dundas_mccarthy}). We are only defining this at the level of $\pi_0$, so the usual difficulties don't intervene, but this is useful to keep in mind. 
\end{rmk}

Since  \cref{prop:base_change_euler_monoid} is somewhat hard to parse, we motivate it by rephrasing the objects and maps involved in the trace in greater generality. 
\begin{itemize}
\item The Morita equivalence
\[
\thh(\Mod^c_A) \xrightarrow{\sim} \thh(A)
\]
is implemented by a base change dual pair, so the last map in \eqref{easy_trace} is  the Euler characteristic of some $(\mc{C}, \mc{D})$-bimodule $\mc{M}$ for appropriate choices of $\mc{C}, \mc{D}, \mc{M}$. 

\item The middle term in \eqref{easy_trace},  $\bigvee_{M \in \Mod^c_A} \Mod^c_A (M, M)$, is the shadow of a bimodule with trivial actions. That is, we've taken the shadow of the bimodule $\Mod^c_A$ where the usual left and right actions have been base changed to be trivial. 
Then the  map
\[
\bigvee_{M \in \Mod^c_A} \Mod^c_A (M, M) \hookrightarrow \thh(\Mod^c_A)
\]
is the map induced on shadows from $\Mod_A^c$ with trivial left and right actions to $\Mod_A^c$ with the usual $\Mod_A^c$ actions. 

If we rephrase this categorically, using the conventions above, $\Mod_A^c$ is $\mc{C}$ and the base change is implemented by a functor $F': \mc{A}' \to \mc{C}$. The module with trivialized action is $ _{F'} \mc{C}_{F'}$ --- note that this is a $(\mc{A}',\mc{A}')$-bimodule.  Then the third map in \eqref{easy_trace} is the map 
\[\sh{_{F'}\mc{C}_{F'}}\to \sh{\mc{C}}\]
induced by $F'$.
By \cref{lem:base_change_euler_ex} this map is $\sh{_{F'} \mc{C}_{F'}} \xrightarrow{\tr\left(\rbm{\epsilon}{F'}{\mc{C}}{U_\mc{C}}\right)} \sh{\mc{C}}$.

\item In the inclusion
\[
\End_{\Mod^c_A} (P) \hookrightarrow \bigvee_{M} \End_{\Mod^c_A}(M) 
\]
the module $\End_{\Mod^c_A} (P)$ is a shadow obtained by trivializing more of the action. Phrasing this categorically, $\End_{\Mod^c_A} (P)$ is $_F \mc{C}_F$ where $F$ is a composite $\mc{A} \to \mc{A}' \to \mc{C}$. 
\item Finally
\[
S \to \End_{\Mod^c_A} (P)
\]
is simply a map of spectra (i.e. $S$-modules). In the more categorical setting, this is a map of $(\mc{A}, \mc{A})$-bimodules $\mc{Q} \to{ _F \mc{C}_F}$.
\end{itemize}

To summarize, if we have an $(\mc{C}, \mc{D})$-bimodule $\mc{M}$, a functor $F: \mc{A} \to \mc{C}$ and a map of $(\mc{A}, \mc{A})$-bimodules $\mc{Q} \to {_F \mc{C}_F}$, we have a map 
\[
\sh{\mc{Q}} \to \sh{_F \mc{C}_F} \xto{\tr\left(\rbm{\epsilon}{F}{\mc{C}}{U_\mc{C}}\right)} \sh{\mc{C}} \xrightarrow{\chi(\mc{M})} \sh{\mc{D}}. 
\]
The first map comes from the 2-functoriality of shadows.   
Having placed ourself in a purely category theoretic context, we may do this nearly trivially. Many illuminating examples arise from specializations of the category theory including, of course, the main example.

\begin{defn}
If $\mc{M}$ is an $(\mc{C},\mc{D})$-bimodule  and $F\colon \mc{A}\to \mc{C}$ is a functor 
let $\rbmm{\epsilon}{F}{\mc{C}}{{\mc{C}}}{\mc{M}}$ be the $(\mc{A},\mc{D})$-bimodule map given by  composite 
\[_F\mc{C}_F\odot {_F\mc{M}}\xleftarrow{\id_{_F\mc{C}_F}\odot \sim } {_F\mc{C}_F}\odot {_F\mc{C}}\odot \mc{M} \xto{\rbm{\epsilon}{F}{\mc{C}}{U_{\mc{C}}}\odot \id_\mc{M}} {_F\mc{C}}\odot \mc{M}
\xto{\sim } {_F\mc{M}}\]
This map is induced by the action of $\mc{C}$ on $\mc{M}$ restricted to the objects in the image of $F$.
\end{defn}

\begin{prop}\label{prop:base_change_euler_monoid}
If $\mc{M}$ is a right dualizable $(\mc{C},\mc{D})$-bimodule  and $F\colon \mc{A}\to \mc{C}$ is an enriched functor, the trace of $\rbmm{\epsilon}{F}{\mc{C}}{{\mc{C}}}{\mc{M}}$ is the composite 
\[\sh{{_F\mc{C}_F}}\xto{\tr\left(\rbm{\epsilon}{F}{\mc{C}}{U_\mc{C}}\right)}
\sh{\mc{C}}\xto{\chi(\mc{M})} \sh{\mc{D}}.\]
If $\alpha \colon \mc{Q}\to {_F\mc{C}_F}$ is a map of $(\mc{A},\mc{A})$-{\bf bimodules} 
 the trace of 
$\mc{Q}\odot {_F\mc{M}}\xto{\alpha \odot \id_{{_F\mc{M}}}}{_F\mc{C}_F}\odot {_F\mc{M}}\xto{\rbmm{\epsilon}{F}{\mc{C}}{{\mc{C}}}{\mc{M}} } 
{_F\mc{M}}$
is 
\[\sh{\mc{Q}}\xto{\sh{\alpha}} \sh{ {_{F}\mc{C}_{F}}}\xto{\tr\left(\rbm{\epsilon}{F}{\mc{C}}{U_\mc{C}}\right)}\sh{\mc{C}}\xto{\chi(\mc{M})} \sh{\mc{D}}. \]
\end{prop}

\begin{proof}Using \cref{lem:composite_dual_pair} $\chi(\mc{M})\circ \tr\left(\rbm{\epsilon}{F}{\mc{C}}{U_\mc{C}}\right)$ is the trace of the top row of the following commutative diagram.
\[\xymatrix@C=60pt{
_F\mc{C}_F\odot {_F\mc{C}}\odot \mc{M}\ar[d]\ar[r]^-{\rbm{\epsilon}{F}{\mc{C}}{U_\mc{C}}\odot \id_\mc{M}}\ar[d]
&  {_F\mc{C}}  \odot \mc{M}\ar[d]
\\
_F\mc{C}_F\odot {_F\mc{M}}\ar[r]^-{\rbmm{\epsilon}{F}{\mc{C}}{\mc{C}}{\mc{M}}}
& {_F\mc{M}}
}\] 
Since the vertical maps are identity maps on twisting objects and an isomorphism on the dualizable object   
the trace of $\rbm{\epsilon}{F}{\mc{C}}{U_\mc{C}}\odot \id_{\mc{M}}$ and $\rbmm{\epsilon}{F}{\mc{C}}{\mc{C}}{\mc{M}}$ are the same.

If $\mc{Q}$ is an $(\mc{A},\mc{A})$-bimodule and $\alpha\colon \mc{Q}\to {_F\mc{C}_F}$ is a module 
homomorphism, then \cref{prop:tightening}  implies 
\[\sh{\mc{Q}}\xto{\sh{\alpha}} \sh{_F\mc{C}_F}\xto{\tr\left(\rbmm{\epsilon}{F}{\mc{C}}{{\mc{C}}}{\mc{M}}\right)} \sh{\mc{D}}\]
is the trace of the composite 
$\mc{Q}\odot {_F\mc{M}} \xto{\alpha\odot \id_{_F\mc{M}}} {_F\mc{C}_F}\odot {_F\mc{M}}\xto{\rbmm{\epsilon}{F}{\mc{C}}{{\mc{C}}}{\mc{M}}}
{_F\mc{M}}$.
\end{proof}

Before we consider examples of this result we need to fix some notation and give an example of \cref{lem:base_change_euler_ex}.
If $\mc{C}$ is  a category enriched  in $\mc{V}$, let $1_{\mc{C}}$ be the category whose objects are the objects of $\mc{C}$ and 
\[1_{\mc{C}}(c,c')=\left\{\begin{array}{cc} 1_\mc{V}&\text{if }c=c'
\\
\emptyset &\text{if } c\neq c'\end{array}\right.\]
The composition map is the unit isomorphism.
Let
$I\colon 1_\mc{C}\to \mc{C}$ be the functor that picks out the identity map for each $c\in \mc{C}$.
If $c$ is an object of $\mc{C}$, and we regard $1_{\mc{V}}$ as a one object category with object $\ast$,  there is a functor $I_c\colon 1_\mc{V}\to 1_\mc{C}\to \mc{C}$ that  picks out the identity map of $c$.
As in \cref{lem:base_change_euler_ex}, 
the following diagrams commute. 
\begin{equation}\label{ex:connect_shadows_endo}
\xymatrix{
\coprod_{c} \mc{C}(I(c),I(c))\ar@{=}[r]
&\coprod_{c} \mc{C}(c,c)\ar[d]
\\
\sh{{_I(U_\mc{C})_I}}\ar[r]^-{\tr\left(\rbm{\epsilon}{I}{\mc{C}}{U_\mc{C}}\right)}\ar@{=}[u]
&\sh{\mc{C}}
}\quad \text{and}\quad \xymatrix{
 \mc{C}(I_c(\ast),I_c(\ast))\ar[r]
&\coprod_{c} \mc{C}(c,c)\ar[d]
\\
\sh{{_{I_c}(U_\mc{C})_{I_c}}}\ar[r]^-{\tr\left(\rbm{\epsilon}{I_c}{\mc{C}}{U_\mc{C}}\right)}\ar@{=}[u]
&\sh{\mc{C}}
}\end{equation}
These commuting squares provide the necessary connection between maps of endomorphisms and traces.

\begin{example}\label{ex_ab_group_base_change}
If $\mc{V}$ is the category of abelian groups and categories $\mc{C}$ and $\mc{D}$ each have  a single object, $\mc{C}$ is a ring $C$,  $\mc{D}$ is a ring $D$.  An $(\mc{C}, \mc{D}$)-bimodule $\mc{M}$ is an $(C,D)$-bimodule $M$.  
If $i\colon \Z\to C$ is the map that picks out the monoidal unit, 
\[\rbm{\epsilon}{i}{C}{C}\colon _iC_i\otimes{ _iC}\to{_iC}\] is the ring multiplication in $C$ and 
\[\rbmm{\epsilon}{i}{C}{C}{M}\colon _iC_i\otimes{ _iM}\to {_iM} \] is the module structure map.  
If $f\colon A\to C$ is a ring homomorphism, and $M$ is an $(C,D)$-bimodule 
\[\rbmm{\epsilon}{i}{A}{A}{_fM}\colon _iA_i\otimes{ _i({_fM})}\to {_i({_fM})} \] 
is given by 
$\rbmm{\epsilon}{i}{A}{A}{_fM}(a,m)=f(a)m$. 

\cref{prop:base_change_euler_monoid} implies the trace of the module structure map $\rbmm{\epsilon}{i}{C}{C}{M}\colon _iC_i\otimes{ _iM}\to {_iM}$ is 
the composite 
\[\sh{_iC_i}\to \sh{C}\xto{\chi(M)} \sh{D}\]
The commuting diagram in \eqref{ex:connect_shadows_endo} identifies the first map as the map $C\to \sh{C}$.

If $\alpha \colon \bZ\to {_iC_i}$ is a module homomorphism,
the composite 
\[M\cong \bZ\otimes {_iM}\xto{\alpha\otimes \id} {_iC_i}\otimes {_iM} \xto{\rbmm{\epsilon}{i}{C}{C}{M}} {_iM}\] is given by $n\mapsto \alpha(1)n$.
While this is very similar to the description of $\rbmm{\epsilon}{i}{A}{A}{_fM}$, since $\alpha$ is a module homomorphism there is greater flexibility in the image of $1$.

\cref{prop:base_change_euler_monoid} implies the trace of 
\[\Z\odot {_iM}\xto{\alpha \odot \id_{_iM}}{_iC_i}\odot {_iM}\xto{\rbmm{\epsilon}{i}{C}{C}{M} } {_iM}\]
is $\sh{\Z}\xto{\sh{\alpha}} \sh{ _iC_i}\xto{\tr(\rbm{\epsilon}{i}{C}{C})}\sh{C}\xto{\chi(M)} \sh{D}$.

\end{example}

While we focused on the case where $\mc{V}$ is the category of abelian groups, \cref{ex_ab_group_base_change} holds as long as objects of $\mc{V}$ have underlying sets.

\begin{cor}
For a functor  $F\colon \mc{A}\to \mc{C}$, a right dualizable $(\mc{C},\mc{D})$-bimodule $\mc{M}$, 
an object $a\in \mc{A}$ and a map of $(1_\mc{V},1_\mc{V})$-{bimodules} 
 $\alpha \colon 1_\mc{V}\to {_{I_a}\mc{A}_{I_a}}$
the trace of 
\[1_\mc{V}\odot {_{F\circ I_a}\mc{M}}\xto{\alpha \odot \id_{{_{F\circ I_a}\mc{M}}}} 
{_{I_a}\mc{A}_{I_a}}\odot {_{F\circ I_a}\mc{M}}
\cong 
{_{I_a}\mc{A}_{I_a}}\odot {_{I_a}({_{F}\mc{M}})}
\xto{\rbmm{\epsilon}{I_a}{\mc{A}}{U_\mc{A}}{_F\mc{M}}}
 {_{I_a}({_{F}\mc{M}})}\cong {_{F\circ I_a} \mc{M}}\]
as a map of $(1_{\mc{V}},\mc{D})$-bimodules is 
\[{1_\mc{V}\xto\alpha
\sh{{_{I_a}\mc{A}_{I_a}}}\xto{\tr(\rbm{\epsilon}{I_a}{\mc{B}}{U_\mc{D}})}
\sh{\mc{A}}\xto{\chi(_F\mc{M})} \sh{\mc{D}}.
}\]
\end{cor}
\begin{proof}
This follows from \cref{prop:base_change_euler_monoid} with the substitutions
\begin{center}
\begin{tabular}{c|c|c|c|c|c}
$\mc{A}$
&$\mc{C}$
&$\mc{D}$
&$\mc{M}$
&$F$
&$\mc{Q}$
\\
\hline
$1_\mc{V}$
&
$\mc{A}$
&$\mc{D}$
&${_F\mc{M}}$
&$I_a$
&$1_\mc{V}$
\end{tabular}
\end{center}
\end{proof}

\begin{thm} \label{lem:base_change_euler_inverse}
For a Morita equivalence $E_a\colon \End_{\mc{A}}(a)\to \mc{A}$, an object $b$ of $\mc{A}$,  and a map of modules $\alpha\colon 1_{\mc{V}}\to  {_{I_b}\mc{A}_{I_b}}$, the trace of \[1_\mc{V}\odot  {_{I_b} \mc{A}_{E_a}} \xto{\alpha \odot \id} {_{I_b}\mc{A}_{I_b}}\odot  {_{I_b}\mc{A}_{E_a}}  \xto{\rbmm{\epsilon}{I_b}{\mc{A}}{U_\mc{A}}{\mc{A}_{E_a}}}  {_{I_b}\mc{A}_{E_a}} \]
is
the composite 
\[1_{\sV}\xto{\alpha}\mc{A}(b,b)
\to  \sh{\mc{A}}\xto{\chi_{(\mc{A}_{E_a},{_{E_a}\mc{A}})} (\mc{A}_{E_a})} \sh{\End_{\mc{A}}(a)}.\]
 \end{thm}

\begin{proof}
This follows from  \cref{prop:base_change_euler_monoid} with the substitutions
\begin{center}
\begin{tabular}{c|c|c|c|c|c}
$\mc{A}$
&$\mc{C}$
&$\mc{D}$
&$\mc{M}$
&$F$
&$\mc{Q}$
\\
\hline
$\End_\mc{A}(b)$
&$\mc{A}$
&$\End_\mc{A}(a)$
&$ \mc{A}_{E_a}$
&$I_b$
&$1_\mc{V}$
\end{tabular}
\end{center}
\end{proof}

\begin{example}\label{ex:base_change_euler_monoid_rings}
  Let $\mc{V}$ be the category of abelian groups. For a ring $C$ and a finitely generated projective $C$-module $P$, 
an endomorphism $f\colon P \to P$ is represented by a map of abelian groups $\Z \xto{[f]} \Hom(P, P)$, i.e. a map of $\Z$-modules. Despite the fact that $\Hom(P, P)$ has the structure of a monoid, the map is \textit{not} a map of monoids. 

Applying \cref{lem:base_change_euler_inverse} with the substitutions
\begin{center}
\begin{tabular}{c|c|c|c}
$\mc{A}$
&$a$
&$b$
&$\alpha$
\\
\hline
$\Mod_C^c$
&$C$
&$P$
&$[f]$
\end{tabular}
\end{center}
the composite 
\[\Z\xto{[f]}\Mod_C^c(P,P)
\to  \sh{\Mod_C^c}\xto{\chi_{\left((\Mod_C^c)_{E_C},{_{E_C}(\Mod_C^c)}\right)} \left((\Mod_C^c)_{E_C}\right)} \sh{\End_{\Mod_C^c}(C)}\]
is
the trace of 
\begin{equation}\label{eq:dual_f_map}
\Mod_C^c(P,C) \cong \Z\otimes  \Mod_C^c(P,C)  \xto{[f] \odot \id} \End_{\Mod_C^c}(P) \otimes \Mod_C^c(P,C)   \xto{\rbmm{\epsilon}{I_P}{\Mod_C^c}{U_{\Mod_C^c}}{(\Mod_C^c)_{E_C}}}  \Mod_C^c(P,C) .
\end{equation}
Following \cref{ex_ab_group_base_change} the second map is the action of $\End_{\Mod_C^c}(P) $ on $\Mod_C^c(P,C) $ by composition, and so the 
the image of $\psi\in \Mod_C^c(P,C) $ in $\Mod_C^c(P,C) $  is $\psi\circ f$. So \eqref{eq:dual_f_map}
 is the dual of $f$ and the trace of  \eqref{eq:dual_f_map} is the trace of $f$ as in \cref{left_or_right_for_trace}.
\end{example}

With \cref{lem:base_change_euler_inverse} in place, we can state one of the main theorems of this paper, concerning the relationship between the cyclotomic trace and the bicategorical trace. 

\begin{thm}\label{thm:pi_0_THH}
  Let $R$ be a ring spectrum, and $f\colon M \to M$ an endomorphism of compact $R$-modules. Then the composition
  \[
  S \xrightarrow{f} \operatorname{Mod}^c_R (M, M) \hookrightarrow \bigvee_M \operatorname{Mod}^c_R(M, M) \hookrightarrow \thh (\operatorname{Mod}^c_R) \xrightarrow{\chi} \thh(R)
  \]
  coincides with the bicategorical trace $S \to \thh(R)$ induced by $f\colon M \to M$. 
\end{thm}

\begin{proof}
We use \cref{lem:base_change_euler_inverse} with the substitutions
\begin{center}
\begin{tabular}{c|c|c|c|c}
 $\mc{A}$ 
&$a$
&$b$ 
&$1_\mc{V}$
&$\alpha$
\\
\hline
 $ \operatorname{Mod}^c_R$ 
&$R$ 
&$M$ 
&$S$
&$f$. 
\end{tabular}
\end{center}
\end{proof}

\appendix

\section{Identifying $\mc{L} X^f$ via $\thh$}\label{sect:identify_thh}

This appendix is devoted to performing a computation that is necessary for \cref{thm:main_theorem}. We note that this computation is not strictly necessary since \cref{main_geometric_theorem} can be deduced from structural results. However, the inclusion of this exposition is useful for understanding geometrically \textit{why} the twisted free loop space should be appearing. 

Given a map $f\colon  X \to X$ we identify the twisted $\thh$ spectrum $\thh(\Sigma^\infty_+ \Omega X; \Sigma^\infty_+ \Omega X_f)$ with a twisted free loop space $\mc{L}X^f$. It is probable that computations of this sort for arbitrary self-maps of ring spectra, $g\colon  A \to A$, are interesting; but we content ourselves with the current example. 

It is a classical fact due to Goodwillie \cite{goodwillie_cyclic} that the cyclic bar construction applied to the based loop space is the free loop space: $B^{\text{cy}} (\Omega X) \simeq \mc{L} X$. In modern categories of spectra, the suspension spectrum functor interacts nicely with bar constructions \cite{EKMM}, and so this provides a computation of the topological Hochschild homology of the ``spectral group ring'' $\Sigma^\infty_+ \Omega X$:
\[
\thh (\Sigma^\infty_+ \Omega X) \simeq \Sigma^\infty_+ \mc{L} X 
\]
In the bar construction above, $\Omega X$ is considered as an $(\Omega X, \Omega X)$-bimodule in the obvious way. In what follows, we consider $\Omega X$ as an $(\Omega X, \Omega X)$-bimodule with the action twisted by an endomorphism $f\colon  X \to X$. That is, let $\omega \in \Omega X$ and let $\gamma$ be the path from $\ast$ to $f(\ast)$, then we define the left action of $\omega ' \in \Omega X$ by $\omega' \ast \omega$ and the right action by $\omega \ast \gamma \ast f(\omega') \ast \gamma^{-1}$. Let $\Omega X^f$ be $\Omega X$ with this bimodule structure; we may then ask about $B^{\text{cy}} (\Omega X; \Omega X^f)$. We compute this below and show it to be homotopy equivalent to a twisted version of the free loop space, which is the proper receptacle for the Reidemeister trace.

The computation proceeds mostly as in \cite{goodwillie_cyclic}. We compare the cyclic bar construction as a simplicial space with the singular simplicial space of the twisted free loop space. In order to work with strict topological monoids and bimodules we need to work with Moore path spaces, Moore loop spaces, etc. Also, the introduction of $\gamma$ above in order to transport between loops based at $\ast$ and loops based at $f(\ast)$ is unfortunate and we avoid it below by choosing two different models for $\Omega X$ when it suits us.

\begin{defn}
  The \textbf{free Moore path space} is defined to be 
  \[
  \mc{P}^M X = \{(\gamma, u) \in \operatorname{Map}([0, \infty), X) \times [0, \infty) : \gamma(t) = \gamma(u) \ t \geq u \}.
      \]
     Here $\operatorname{Map}([0,\infty), X)$ is given the compact open topology. The space $\mc{P}^M X$ comes equipped with two maps $s\colon  \mc{P}^M X \to X$ and $t\colon  \mc{P}^M X \to X$ given by evaluation at $0$ and $u$. 
\end{defn}

\begin{defn}
  Let $X$ be a based space with base $\ast$. Then the \textbf{Moore loop space} is 
  \[
  \Omega^M X = \{(\gamma, u) \in \mc{P}^M X : \gamma(0) = \gamma(u) = \ast\} 
  \]
  that is, it is the pullback
  \[
  \xymatrix{
    \Omega^M X \ar[r]\ar[d] & \mc{P}^MX \ar[d]^{s \times t} \\
    \ast\ar[r]  & X \times X 
  }
  \]
  For work below, it is convenient to denote the ``length'' of the path, $u$, by $|\omega|$. 
\end{defn}
For use in the two-sided bar construction, we need another version of this loop space. The following construction allows us to act on the left by loops based at $\ast$ and on the right by loops based at $f(\ast)$.

\begin{defn}\label{defn:twisted_loop_space}
  The \textbf{$f$-twisted Moore loop space} is the topological space 
  \[
  \Omega^M X^f = \{(\gamma, u) \in \mc{P}^M X : \gamma(0) = \ast, \gamma(u) = f(\ast)\}
  \]
  i.e. it is given as a pullback
  \[
  \xymatrix{
    \Omega^M X^f \ar[r]\ar[d] & \mc{P}^M X \ar[d]\\
    \ast \ar[r]^{\operatorname{id} \times f} & X \times X 
  }
  \]
  Again, $|\omega|$ denotes the length of the path $\omega$. 
\end{defn}

  This space $\Omega^M X^f$ is homotopy equivalent to $\Omega^M X$ and has the structure of an \\$(\Omega^M X, \Omega^M X)$-bimodule. For $\alpha \in \Omega^M X$, $\gamma\in \Omega^M X^f$ and $\beta \in \Omega^M X$ 
  \[
  \alpha \cdot \gamma = \alpha \ast \gamma \qquad \gamma \cdot \beta = \gamma \ast f(\beta) 
  \]

  For our definition of the free loop space, it is more convenient to use the naive definition, not the Moore-style definition.  We let $\mc{P}X$ be paths in $X$ and 
$\Omega X$ be loops in $X$.
\begin{defn}
  The \textbf{$f$-twisted free loop space} is defined to be
  \[
  \mc{L} X^f = \{\gamma \in \mc{P} X : \gamma(0) = f(\gamma(1))\} 
  \]
  That is, $\mc{L} X^f$ arises as the pullback
  \[
  \xymatrix{
    \mc{L} X^f \ar[d]\ar[r] & \mc{P} X \ar[d]^{\operatorname{ev}_0 \times \operatorname{ev}_1}\\
    X \ar[r]^{\operatorname{id} \times f} & X \times X 
  }
  \] 
  Consequently, $\mc{L}^f X$ sits in a fibration sequence
  \[
  \Omega X^f \to \mc{L}^f X \to X 
  \]
\end{defn}

We can now define the main object of concern in this section. 

\begin{defn}
  $N^{\text{cy}} (\Omega^M X, \Omega^M X^f)$ is the geometric realization of the simplicial space with $n$-simplices
  \[
  N^{\text{cy}}_n (\Omega^M X, \Omega^M X^f) = \underbrace{\Omega^M X \times \cdots \times \Omega^M X}_{n\ \text{times}} \times \Omega^M X^f
  \]
  The face maps $d_i \colon  N^{\text{cy}}_n \to N^{\text{cy}}_{n-1}$ are given by 
  \[
  d_i (\omega_0, \dots, \omega_{n-1}, \omega) =
  \begin{cases}
    (\omega_0, \dots, \omega_i \omega_{i+1}, \dots, \omega_{n-1}, \omega) & i < n \\
    (\omega_0, \dots, \omega_{n-1} \omega) & i = n \\
    (\omega_1, \dots, \omega_{n-1}, \omega f(\omega_0)) & i = n+1
  \end{cases}
  \]
  The degeneracies $s_j \colon  N^{cy}_{n} \to N^{\text{cy}}_{n+1}$ are given by the insertion of trivial paths:
  \[
    s_j (\omega_0, \dots, \omega_{n-1}, \omega) = (\cdots, \omega_{j}, \ast, \omega_{j+1}, \dots)
  \]
\end{defn}

\begin{rmk}
Level-wise this is equivalent to the cyclic bar construction that computes $\mc{L} X$, but the face maps are different. 
\end{rmk}

In order to proceed with the comparison, we need to define various simplicial spaces. In what follows elements in $\Delta^n$ referred to by the barycentric coordinates $(u_0, \dots, u_n)$. 

The $\Delta^n$ are assembled into a cosimplicial space $\Delta^\bullet$ via the cosimplicial maps
\begin{align*}
d_i (u_0, \dots, u_n) &= (u_0, \dots, u_i + u_{i+1}, \dots, u_n)\\
s_i (u_0, \dots, u_n) &= (u_0, \dots, u_i, 0, u_{i+1}, \dots, u_n) 
\end{align*}

For $\gamma\in \operatorname{Map}(\Delta^n, \mc{P} X) $ and $( u_0, \dots, u_n)\in \Delta^n$, $\gamma_{u_0, \dots, u_n}$ denotes the result of evaluating 
$\gamma$ at $( u_0, \dots, u_n)$.

\begin{defn}
  The space $(\Omega X^f)^{\Delta^n}$ is defined to be 
  \[
  (\Omega X^f )^{\Delta^n} = \left\{\gamma \in \operatorname{Map}(\Delta^n, \mc{P} X) \biggr| \begin{matrix}\gamma_{u_0, \dots, u_n} (0) = \ast, \gamma_{u_0, \dots, u_n}(1) = f(\ast)\\ \text{for all}\ (u_0, \dots, u_n) \in \Delta^n \end{matrix}\right\}
  \]
  The cosimplicial structure on $\Delta^\bullet$ induces a simplicial structure on $(\Omega X^f)^{\Delta^\bullet}$.  The face maps are given by restriction:
  \[
  d_i (\gamma)(t, u_0, \dots, u_n) = \gamma(t, u_0, \dots,u_i, 0, u_{i+1}, \dots, u_n) 
 \]
  where $t$ is the path coordinate and the degeneracy maps are given by addition of coordinates:
  \[
  s_i (\gamma)(t, u_0, \dots, u_n) = \gamma(t, u_0, \dots, u_i + u_{i+1}, \dots, u_n) 
  \]
\end{defn}

\begin{defn}
 The topological space $(\mc{L} X^f)^{\Delta^n}$ is defined to be 
  \[
  (\mc{L} X^f)^{\Delta^n} = \left\{\gamma \in \operatorname{Map}(\Delta^n, \mc{P} X)\biggr|\begin{matrix} \gamma_{u_0, \dots, u_n} (0) = f (\gamma_{u_0, \dots, u_n} (1)) \\ \text{for all}\ (u_0, \dots, u_n) \in \Delta^n \end{matrix}\right\} 
  \]
Again, the cosimplicial structure on $\Delta^\bullet$ induces a simplicial structure on $(\mc{L} X^f)^{\Delta^\bullet}$. The face and degeneracy maps are as above. 
\end{defn}

The idea is now to compare the bar constructions with known fibration sequences.  Consider the following diagram.

\begin{equation}\label{main_diagram}
\xymatrix{
  (\Omega^M X^f)_{\bullet, c} \ar[r]^-{A}\ar[d]_{i} & (\Omega X^f)^{\Delta^\bullet} \ar[d]^{i^{\Delta^\bullet}}\\
  N^{\text{cy}}(\Omega^M X, \Omega^M X^f)\ar[d]_{p}\ar[r]^-{B} & (\mc{L} X^f)^{\Delta^\bullet}\ar[d]^{p^{\Delta^\bullet}}\\
  B_\bullet \Omega^M X \ar[r]^-{C} & X^{\Delta^\bullet} 
}
\end{equation}

The upper left corner is the constant simplicial space $(\Omega^M X^f)_{\bullet,c}$, and the lower left corner is the usual bar construction on the topological monoid $\Omega^M X$. The right side of the diagram is simply the singular simplicial spaces of the fibration $\Omega X^f \to \mc{L} X^f \to X$.

The left column will be a fibration sequence upon geometric realization, the right hand side is trivially so, and we show that all horizontal maps are weak equivalences following \cite{goodwillie_cyclic}. 

We now write out all of the maps in the diagram:

\begin{itemize}
\item The map
\[
A \colon  (\Omega^M X^f)_{n,c} \to (\Omega X^f)^{\Delta^n} 
\]
is given by reparamterizing the loops:
\[
A (\omega) (t, u_0, \dots, u_n) = \omega(t|\omega|) 
\]

\item 
The map
\[
B \colon  N^{\text{cy}}_{n} (\Omega^M X, \Omega^M X^f) \to (\mc{L} X^f)^{\Delta^n}
\]
is given by:
\begin{align*}
  &B(\omega_0, \omega_1, \dots, \omega_{n-1}, \omega)(t, u_0, \dots, u_n) \\
  &= \widetilde{\omega}\biggl(t (|\omega_0| + |\omega_1| + \cdots + |\omega_{n-1}| + |\omega|)\\
  &+ u_0 (|\omega_0| + \cdots + |\omega_{n-1}|)
  + u_1 (|\omega_1| + \cdots + |\omega_{n-1}|) 
  \cdots
  + u_{n-1} (|\omega_{n-1}|)\biggr)
\end{align*} 
where
\[
\widetilde{\omega} = \omega_0 \ast \cdots \ast \omega_{n-1} \ast \omega \ast f(\omega_0) \ast f(\omega_1) \cdots \cdots \ast f(\omega_{n-1}) 
\]

\item The map
  \[
  C\colon  B_n \Omega^M X \to X^{\Delta^n} 
  \]
  is given by
  \begin{align*}
    &C(\omega_0, \dots, \omega_{n-1})(u_0, \dots, u_n)\\ &= (\omega_1 \ast \cdots \ast \omega_n)\biggl( u_0 (|\omega_0| + \cdots + |\omega_{n-1}|)
    +u_1 (|\omega_1| + \cdots + |\omega_{n-1}|)+
    \cdots+
    u_{n} (|\omega_{n-1}|)\biggr)
  \end{align*}

\item The  maps $i^{\Delta^\bullet}$ and $p^{\Delta^\bullet}$ the obvious ones induced by maps on spaces.
\item The map $i_\bullet$ is given level-wise as
  \[
  i_n (\omega) = (\ast, \dots, \ast, \omega)
  \]
  and $p_\bullet$ is given level-wise as
  \[
  p_n (\omega_0, \dots, \omega_{n-1}, \omega_n) = (\omega_0, \dots, \omega_{n-1}) 
  \]
\end{itemize}

The definition of the map, $B$, above is somewhat elaborate. The example in \cref{map_B} hopefully elucidates it.
\begin{figure}
\begin{tikzpicture}
\coordinate (n11) at (4,0);
\node at (n11){$|$};
\coordinate (n12) at (10,0);
\node at (n12){$|$};
\coordinate (n21) at (2,-1);
\node at  (n21) {$|$};
\coordinate (n22) at (8,-1);
\node at (n22){$|$};
\coordinate (n31) at (0,-2);
\node at  (n31) {$|$};
\coordinate (n32) at (6,-2);
\node at (n32){$|$};
\coordinate (n41) at (0,-3);
\node at (n41){$|$};
\coordinate (n42) at (2,-3);
\node at (n42){$|$};
\coordinate (n43) at (4,-3);
\node at (n43){$|$};
\coordinate (n44) at (6,-3);
\node at (n44){$|$};
\coordinate (n45) at (8, -3);
\node at (n45){$|$};
\coordinate (n46) at (10,-3);
\node at (n46){$|$};
\draw (n11)edge [above] node{$u_0=1$} (n12);
\draw (n21)edge [above] node{$u_1=1$} (n22);
\draw (n31)edge [above] node{$u_2=1$} (n32);
\draw (n41)edge [above] node{$\omega_0$} (n42);
\draw (n42)edge [above] node{$\omega_1$} (n43);
\draw (n43)edge [above] node{$\omega$} (n44);
\draw (n44)edge [above] node{$f(\omega_0)$} (n45);
\draw (n45)edge [above] node{$f(\omega_1)$} (n46);
\draw[decoration={brace,mirror, raise=8pt},decorate]
  (n41) -- node[below=12pt]{$\tilde{\omega}$} (n46);
\coordinate (n51) at (1,-9);
\node at (n51){$|$};
\node at (n51)[above=5pt]{$\ast$};
\coordinate (n52) at (3,-7);
\node at (n52) {$|$};
\node at (n52)[above=5pt]{$\ast$};
\coordinate (n53) at (6,-9);
\node at (n53){$|$};
\node at (n53)[above=5pt]{$\ast$};
\coordinate (n54) at (8,-8.75);
\node at (n54){$|$};
\node at (n54)[above=5pt]{$f(\ast)$};
\coordinate (n55) at (8.5, -6);
\node at (n55){$|$};
\node at (n55)[above=5pt]{$f(\ast)$};
\coordinate (n56) at (6,-5);
\node at (n56){$|$};
\node at (n56)[above=5pt]{$f(\ast)$};
\draw (n51)edge [right, out =0, in =180] node{$\omega_0$} (n52);
\draw (n52)edge [left, out =0, in =180 ] node{$\omega_1$} (n53);
\draw (n53)edge [above, out =0, in =200 ] node{$\omega$} (n54);
\draw (n54)edge [right, out =20, in =-50 ] node{$f(\omega_0)$} (n55);
\draw (n55)edge [above, out =130, in =0 ] node{$f(\omega_1)$} (n56);
\end{tikzpicture}
\caption{The map $B$}\label{map_B}
\end{figure}

\begin{example}
  We consider the case when $n = 2$ for the map $B$ above, which is enough to capture the issues. We define
  \[
  \widetilde{\omega} = \omega_0 \ast \omega_1 \ast \omega \ast f(\omega_0) \ast f(\omega_1)
  \]
  and  a family of paths parameterized by the 2-simplex \[\Delta^2 = \{(u_0, u_1, u_2) : u_0 + u_1 + u_2 = 1\}.\] 
We think of this simplex as interpolating between three extremes: the cases when each of $u_0, u_1$ and $u_2$ are 1. We'd like the extreme cases to be
  \begin{itemize}
  \item $u_2 = 1$ --- the path $\omega_0 \ast \omega_1 \ast \omega$
  \item $u_1 = 1$ --- the path $\omega_1 \ast \omega \ast f(\omega_0)$
  \item $u_0 = 1$ --- the path $\omega \ast f(\omega_0) \ast f(\omega_1)$. 
  \end{itemize}
  We would also like that if the path starts at $x$, the endpoint is $f(x)$. Consider the path 
  \begin{align*}
    \widetilde{\omega} (t(|\omega_0| + |\omega_1| + |\omega|) + u_0 (|\omega_0| + |\omega_1|) + u_1 (|\omega_0|))
  \end{align*}
  We can check that when $u_2 = 1$ then $u_0, u_1 = 0$ and the above gives $\widetilde{\omega}(t(|\omega_0| + |\omega_1|+|\omega|))$ which is exactly $\omega_0 \ast \omega_1 \ast \omega$. When $u_1 = 1$, then the above is $\widetilde{\omega}(t(|\omega_0| + |\omega_1|) + |\omega_0|)$, which is the path $\omega_1 \ast \omega \ast f(\omega_0)$. Finally, when $u_0 = 1$ we get the path $\omega \ast f(\omega_0) \ast f(\omega_1)$.

  We also note that for any choice of $u_0, u_1$, 
  \begin{align*}
    &\widetilde{\omega}(u_0 (|\omega_0| + |\omega_1|) + u_1 (|\omega_0|)\\
    &= \widetilde{\omega} (|\omega_0| + |\omega_1| + |\omega| + u_0 (|\omega_0| + |\omega_1|) + u_1 (|\omega_0|))
  \end{align*}
  essentially by definition, so that each path parameterized by $u_0, u_1, u_2$ is in $\mc{L} X^f$. 
\end{example}

There are a few things we need to check about the maps above. First, we should check that $B$ is actually a map to $(\mc{L} X^f)^{\Delta^\bullet}$ --- that is, check that the endpoints are correct. The argument is the same as in the example. 

\begin{lem}
The map $B$ defined above, is well-defined. 
\end{lem}

Second, it is easy to observe that the diagram actually commutes. Third, it is clear that $A$ and $C$ are simplicial maps. Though it is an irritating exercise, it is easy to check that $B$ is as well. 

\begin{prop}
$ B \colon  N^{\text{cy}} (\Omega^M X, \Omega^M X^f) \to (\mc{L} X^f)^{\Delta^\bullet}$ is a simplicial map.
\end{prop}

The rest of the proof proceeds as in \cite[Sect.~V]{goodwillie_cyclic}. Upon geometric realization, the left hand column of \ref{main_diagram} becomes a fibration up to homotopy by \cite{segal}, as does the right hand column. To prove that $B$ is a weak equivalence, it thus suffice to prove that $A, C$ are weak equivalences. This is done in \cite[Sect.~V]{goodwillie_cyclic}.

All of this work entitles us to the following theorem. 

\begin{thm}\label{main_geometric_theorem}
  Let $X$ be a (connected) topological space. Then
  \[
  B^{\text{cy}}(\Omega^M X, \Omega^M X^f) \simeq \mc{L} X^f 
  \]
\end{thm}

As a corollary when $f = \operatorname{Id}$ we recover Goodwillie's original computation.

\begin{cor}\cite[Sect.~V.1]{goodwillie_cyclic} For a topological space, the geometric realization of the cyclic bar construction on the based loop space is equivalent to the free loop space:
  \[
  B^{\text{cy}}(\Omega X) \simeq \mc{L} X 
  \]
  and for $X = BG$ where $G$ is a finite group
  \[
  B^{\text{cy}} (G) \simeq \mc{L} BG
  \]
\end{cor}

As another corollary, we have a generalization of the classical computation \[\thh (\Sigma^\infty_+ \Omega X) \simeq \Sigma^\infty_+ \mc{L} X.\] 

\begin{cor}\label{cor:identify_thh_loop}
  \[
  \thh(\Sigma^\infty_+\Omega X; \Sigma^\infty_+ \Omega X^f) \simeq \Sigma^\infty_+ \mc{L} X^f 
  \]
\end{cor}

The space $\mc{L} X^f$ is the space of homotopy fixed points of a self-map $f\colon  X \to X$, computed by taking a ``derived intersection'' of $X$ and the image of $X$ under $f$. One could wish to have a similar $\thh$ description of the derived intersection of two maps $f, g\colon  X \to Y$. The following corollary is a more general statement, and easy corollary of the proof of \cref{main_geometric_theorem}. Though we do not use this generality in the paper, it is useful to record for later work. 

\begin{cor}
  Let $f, g\colon  X \to Y$ be self maps and let $\mc{L} Y^{f,g}$ be the homotopy pullback
  \[
  \xymatrix{
    \mc{L} Y^{f,g} \ar[d]\ar[r] & \mc{P} Y \ar[d]^{\text{ev}_0, \text{ev}_1} \\
    X \ar[r]_{f \times g} & Y \times Y
  }
  \]
  Similarly, let $\Sigma^\infty_+ \Omega Y^{f,g}$ be the $(\Sigma^\infty_+ \Omega Y, \Sigma^\infty_+ \Omega Y)$-bimodule $\Sigma^\infty_+ \Omega X$ with left action by $f$ and right action by $g$. Then
  \[
  \thh(\Sigma^\infty_+ \Omega X; \Sigma^\infty_+ \Omega X^{f,g}) \simeq \Sigma^\infty_+ \mc{L} Y^{f,g}.
  \]
\end{cor}
\begin{proof}
The proof is identical to that of \cref{main_geometric_theorem}
\end{proof}

\bibliographystyle{amsalpha2}
\bibliography{fixed_points}

\end{document}